\pdfoutput=1
\documentclass[a4paper,reqno]{amsart}

\usepackage[utf8]{inputenc}
\usepackage[english]{babel}
\usepackage{microtype} 

\usepackage{hyperref} 
\hypersetup{hidelinks}

\usepackage{amsmath,amssymb,amsthm}
\usepackage{thmtools,thm-restate} 

\usepackage[nameinlink]{cleveref} 

\usepackage{mathtools} 
\usepackage{upgreek}
\usepackage{orcidlink}

\usepackage{tikz} 
\usetikzlibrary{cd}
\usetikzlibrary{calc}

\usepackage{multirow}

\usepackage{dsfont} 
\usepackage{pifont} 
\usepackage{nicefrac} 

\usepackage{ifthen}
\usepackage{enumitem}
\setlist[itemize]{leftmargin=\parindent,itemsep=3pt,parsep=2pt}
\setlist[enumerate]{label=\textup{(\roman{*})},labelindent=3pt,leftmargin=*}

\usepackage{subcaption}

\usepackage{suffix} 

\newcounter{i} 
\newtoks\striche 

\makeatletter
\newcommand{\pushright}[1]{\ifmeasuring@#1\else\omit\hfill$\displaystyle#1$\fi\ignorespaces}
\makeatother

\newcommand{\R}{\mathbb{R}}
\newcommand{\N}{\mathbb{N}} 

\newcommand{\C}{\mathbb{C}}

\newcommand{\Lp}[1]{\mathrm{L}^{#1}} 
\newcommand{\sobolev}[1]{\mathrm{W}^{#1}} 
\newcommand{\Lb}{\mathcal{L}_{\mathrm{b}}} 

\newcommand{\iu}{\mathrm{i}} 
\newcommand{\diffd}{\mathrm{d}} 
\newcommand{\dx}[1][x]{\,\diffd#1}

\renewcommand{\vec}[1]{\mathbf{#1}}
\newcommand{\cl}[2][]{\overline{#2}\ifthenelse{ \equal{#1}{} }{}{^{#1}}} 
\newcommand{\conj}[1]{\overline{#1}} 
\newcommand{\quspace}[3][]{#2\big/#3} 
\newcommand{\ball}{\mathrm{B}}

\DeclareMathOperator{\ran}{ran}

\DeclareMathOperator{\supp}{supp}
\DeclareMathOperator{\dom}{dom}

\DeclareMathOperator{\Div}{div}
\renewcommand{\div}{\Div}
\newcommand{\grad}{\nabla}

\DeclareMathOperator{\rot}{curl}

\DeclareMathOperator{\dist}{dist}

\DeclarePairedDelimiter{\set}{\{}{\}}
\DeclarePairedDelimiter{\norm}{\lVert}{\rVert}
\DeclarePairedDelimiter{\abs}{\vert}{\vert}

\DeclarePairedDelimiterX{\dset}[2]{\{}{\}}{#1\,\delimsize\vert\,\mathopen{} #2}
\DeclarePairedDelimiterX{\scprod}[2]{\langle}{\rangle}{#1,#2}
\DeclarePairedDelimiterX{\dualprod}[2]{\langle}{\rangle}{#1,#2}
\DeclarePairedDelimiterX{\sdprod}[2]{\llangle}{\rrangle}{#1,#2} 

\renewcommand{\Re}{\operatorname{Re}}

\newcommand{\dual}{'}

\newcommand{\adjunsymb}{\ast} 

\newcommand{\adjun}[1][1]{%
  \setcounter{i}{1}%
  \striche={\adjunsymb}%
  \loop%
  \ifnum\value{i}<#1%
  \striche=\expandafter{\the\expandafter\striche\adjunsymb}%
  \stepcounter{i}%
  \repeat%
  ^{\the\striche}%
}

\newcommand{\mapping}[4]{%
  \left\{%
    \begin{array}{rcl}%
      #1 &\to & #2, \\
      #3 &\mapsto & #4
    \end{array}%
  \right.%
}

\newcommand{\hermitian}{^{\mathsf{H}}}

\newcommand{\Hspace}{\mathrm{H}}
\newcommand{\hamiltonian}{\mathcal{H}}

\newcommand{\hH}{\hat{\Hspace}}
\newcommand{\cH}{\mathring{\Hspace}}
\newcommand{\cohom}{\mathcal{H}}

\newcommand{\conC}{\mathrm{C}}
\newcommand{\Cc}[1][\infty]{\mathring{\mathrm{C}}^{#1}}

\newcommand{\XH}{\mathcal{X}_{\hamiltonian}}

\newcommand{\boundtr}[1][]{\gamma_{0}\ifthenelse{\equal{#1}{}}{}{\big\vert_{#1}}}
\newcommand{\normaltr}[1][]{\gamma_{\nu}\ifthenelse{\equal{#1}{}}{}{\big\vert_{#1}}}
\newcommand{\tantr}[1][]{\pi_{\tau}\ifthenelse{\equal{#1}{}}{}{\big\vert_{#1}}}
\WithSuffix\newcommand\tantr*[1][]{\pi_{\tau}\ifthenelse{\equal{#1}{}}{}{\vert_{#1}}}
\newcommand{\tantrex}[1][]{\cl{\pi}_{\tau}\ifthenelse{\equal{#1}{}}{}{\big\vert_{#1}}}
\newcommand{\ctantr}[1][]{\mathring{\pi}_{\tau}\ifthenelse{\equal{#1}{}}{}{\big\vert_{#1}}}
\newcommand{\tanxtr}[1][]{\gamma_{\tau}\ifthenelse{\equal{#1}{}}{}{\big\vert_{#1}}}
\newcommand{\ctanxtr}[1][]{\mathring{\gamma}_{\tau}\ifthenelse{\equal{#1}{}}{}{\big\vert_{#1}}}

\newcommand{\cpt}{\overset{\mathsf{cpt}}{\hookrightarrow}}

\newcommand{\diffop}{L_{\partial}}
\newcommand{\diffopad}{L_{\partial}\hermitian}
\newcommand{\Vtau}[1][]{\mathcal{V}_{\tau\ifthenelse{\equal{#1}{}}{}{,#1}}}
\newcommand{\cVtau}[1][]{\mathring{\mathcal{V}}_{\tau\ifthenelse{\equal{#1}{}}{}{,#1}}}
\newcommand{\Vtaud}[1][]{\mathcal{V}_{\tau\ifthenelse{\equal{#1}{}}{}{,#1}}^{\times}}
\newcommand{\cVtaud}[1][]{\mathring{\mathcal{V}}_{\tau\ifthenelse{\equal{#1}{}}{}{,#1}}^{\times}}

\theoremstyle{plain}
\newtheorem{theorem}{Theorem}[section]
\newtheorem{lemma}[theorem]{Lemma}
\newtheorem{proposition}[theorem]{Proposition}
\newtheorem{corollary}[theorem]{Corollary}

\theoremstyle{definition}
\newtheorem{definition}[theorem]{Definition}

\newtheorem{problem}[theorem]{Problem}

\theoremstyle{remark}
\newtheorem{remark}[theorem]{Remark}

\title[Semi-uniform stability of Maxwell's equaitons]{Semi-uniform stabilization of anisotropic Maxwell's equations via boundary feedback on split boundary}
\author[N.~Skrepek]{Nathanael Skrepek\,\orcidlink{0000-0002-3096-4818}}
\email{nathanael.skrepek@math.tu-freiberg.de}
\author[M.~Waurick]{Marcus Waurick\,\orcidlink{0000-0003-4498-3574}}
\email{marcus.waurick@math.tu-freiberg.de}
\date{\today}


\address{TU Bergakademie Freiberg \\
  Institute of Applied Analysis \\
  Akademiestraße 6 \\
  D-09596 Freiberg \\
  Germany}

\begin{document}

\begin{abstract}
  We regard anisotropic Maxwell's equations as a boundary control and observation system on a bounded Lipschitz domain. The boundary is split into two parts: one part with perfect conductor boundary conditions and the other where the control and observation takes place.  We apply a feedback control law that stabilizes the system in a semi-uniform manner without any further geometric assumption on the domain. This will be achieved by separating the equilibriums from the system and show that the remaining system is described by an operator with compact resolvent. Furthermore, we will apply a unique continuation principle on the resolvent equation to show that there are no eigenvalues on the imaginary axis.
\end{abstract}

\maketitle

\section{Introduction}%
\label{sec:introduction}
Let $\Omega \subseteq \R^{3}$ be a bounded and connected strongly Lipschitz domain, which boundary is split into $\Gamma_{0}$ and $\Gamma_{1} \neq \emptyset$. Then we regard Maxwell's equations as a boundary control and observation system
\begin{subequations}
  \label{eq:maxwells-eq}
  \newcommand{\myvspace}{1ex} 
  \begin{alignat}{3}
    u(t,\zeta) &= \tantr  \vec{E}(t,\zeta), &\mspace{25mu}&\text{(boundary input)} &\mspace{25mu}& t\geq 0, \zeta \in \Gamma_{1}, \\[\myvspace]
    \tfrac{\partial}{\partial t} \vec{D}(t,\zeta) &= \rot \vec{H}(t,\zeta),
                                            &&\text{(Faraday/Maxwell law)}&&t\geq 0, \zeta \in \Omega, \label{eq:maxwells-eq-dynamic-1}
    \\
    \tfrac{\partial}{\partial t} \vec{B}(t,\zeta) &= -\rot \vec{E}(t,\zeta), &&\text{(Amp{\'e}re/Maxwell law)}&&t\geq 0, \zeta \in \Omega, \label{eq:maxwells-eq-dynamic-2}
    \\[\myvspace]
    \div \vec{D}(t,\zeta) &= \rho(\zeta), && \text{(Gau{\ss} law)}&&t\geq 0, \zeta \in \Omega,\\
    \div \vec{B}(t,\zeta) &= 0, && \text{(Gau{\ss} law for magnetism)}&&t\geq 0, \zeta \in \Omega,\\[\myvspace]
    \vec{D}(t,\zeta) &= \epsilon(\zeta) \vec{E}(t,\zeta), && \text{(material law)}&&t\geq 0, \zeta \in \Omega,\\
    \vec{B}(t,\zeta) &= \mu(\zeta) \vec{H}(t,\zeta), && \text{(material law)}&&t\geq 0, \zeta \in \Omega, \\[\myvspace]
    \tantr \vec{E}(t,\zeta) &= 0, &&\text{(perfect conductor)}&&t\geq 0, \zeta \in \Gamma_{0}, \\
    \normaltr \vec{B}(t,\zeta) &= 0, &&\text{(normal boundary cond.)} && t\geq 0, \zeta \in \Gamma_{0}, \\[\myvspace]
    \vec{E}(0,\zeta) &= \vec{E}_{0}(\zeta), &&\text{(initial value)}&& \phantom{t\geq 0,\mathclose{}} \zeta \in \Omega, \\
    \vec{H}(0,\zeta) &= \vec{H}_{0}(\zeta), &&\text{(initial value)}&& \phantom{t\geq 0,\mathclose{}} \zeta \in \Omega, \\[\myvspace]
    y(t,\zeta) &= \tanxtr \vec{H}(t,\zeta), &&\text{(boundary output)}&& t\geq 0, \zeta \in \Gamma_{1}, \\
    \intertext{with feedback law}
    u(t,\zeta) &= -k(\zeta) y(t,\zeta), && &&t\geq 0,\zeta \in \Gamma_{1}.
  \end{alignat}
\end{subequations}
The traces $\tantr \vec{E}$, $\tanxtr \vec{H}$ and $\normaltr \vec{B}$ are, roughly speaking, $(\nu \times \vec{E}\big\vert_{\partial\Omega}) \times \nu$, $\nu \times \vec{H}\big\vert_{\partial\Omega}$ and $\nu \cdot \vec{B} \big\vert_{\partial\Omega}$, respectively, for details see \Cref{sec:background-diffoperator-and-traces}.
The permittivity $\epsilon$ and the permeability $\mu$ are Lipschitz continuous matrix-valued functions (i.e., we allow anisotropic and inhomogeneous materials) such that $c^{-1} \leq \epsilon \leq c$ and $c^{-1} \leq \mu \leq c$ for a $c >0$ (in the sense of positive definitness).
(The Lipschitz continuity of $\epsilon$ and $\mu$ is necessary to apply a unique continuation principle.)
The feedback operator $k$ is also matrix-valued and satisfies $c^{-1} \leq k \leq c$ for a $c >0$ (w.l.o.g.\ the same $c$)---we do not ask for any further regularity but measurability. The charge density $\rho$ can be any $\Lp{2}(\Omega)$ function.

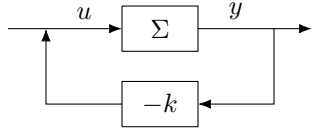
\begin{figure}[h]\centering
  \begin{tikzpicture}
    \draw[-Latex] (-0.5,0) -- (1,0);
    \draw (1,0.3) -- (1,-0.3) -- (2,-0.3) -- (2,0.3) -- (1,0.3);
    \draw[-Latex] (2,0) -- (3.5,0);

    \draw[-Latex] (3,0) -- (3,-1) -- (2,-1);
    \draw (1,-1.3) -- (1,-0.7) -- (2,-0.7) -- (2,-1.3) -- (1,-1.3);
    \draw[-Latex] (1,-1) -- (0,-1) -- (0,0);

    \node at (1.5,-1) {$-k$};

    \node at (1.5,0) {$\Sigma$};
    \node[above] at (0.5,0) {$u$};
    \node[above] at (2.5,0) {$y$};
  \end{tikzpicture}
  \caption{Feedback illustration}
  \label{fig:feedback}
\end{figure}

The boundary feedback law illustrated in \Cref{fig:feedback} results in a \emph{Silver--M{\"u}ller} or \emph{Leontovich} or \emph{impedance} \emph{boundary condition}\footnote{Silver--M{\"u}ller boundary conditions can sometimes have different forms, but by inverting and/or unitarily transforming $k$ they are equivalent.}
\begin{equation*}
  \tantr \vec{E}(t,\zeta) + k(\zeta) \tanxtr \vec{H}(t,\zeta) = 0, \qquad t \geq 0, \zeta \in \Gamma_{1}.
\end{equation*}

Our goal is to show that the system \eqref{eq:maxwells-eq} is \emph{semi-uniformly stable}, i.e., there exists a uniform decay rate such that every solution converges to an equilibrium state with this rate. More precisely this means that the corresponding semigroup is \emph{semi-uniformly stable}, see \Cref{def:semi-uniform-stability}. Semi-uniform stability is a concept that was introduced in \cite{BaDu08}. It is a stability concept that is between strong stability and exponential stability, i.e., exponential stability implies semi-uniform stability and semi-uniform stability implies strong stability.

Stability of Maxwell's equations with Silver--M{\"u}ller boundary conditions have been studied in several works. The goal was always to prove exponential stability and therefore additional assumptions were added.

\begin{itemize}
  \item In \cite{Ko94} the author regarded a domain $\Omega$ with $\conC^{1}$ boundary and assumed $\epsilon = \mu = 1$. The damping acts on the entire boundary.

  \item In \cite{Ph00} $\Omega$ has a $\conC^{\infty}$ boundary and $\Gamma_{1}$ satisfies the geometric control condition. The author worked with constant and scalar $\epsilon$ and $\mu$.

  \item In \cite{ElLaNi02a} the authors regard a domain $\Omega$ with $\conC^{\infty}$ boundary and $\epsilon$ and $\mu$ are scalar $\conC^{\infty}$ functions.
        The damping acts on the entire boundary and they additionally assume the existence of a $\zeta_{0} \in \Omega$ such that
        \begin{align*}
          (\zeta - \zeta_{0}) \grad \epsilon  \geq 0 \quad\text{and}\quad (\zeta - \zeta_{0}) \grad \mu  \geq 0 \quad\text{in}\quad \Omega.
        \end{align*}
        However, they allow $k$ to be in a class that also contains certain nonlinear operators.

  \item In \cite{NiPi03} $\Omega$ has a $\conC^{2}$ boundary and the damping acts on the entire boundary $\partial\Omega$. The functions $\epsilon$ and $\mu$ are scalar $\conC^{1}$ and may be non-autonomous. The function $k$ can be in a certain nonlinear class.

  \item In \cite{AnPo19} $\Omega$ has a $\conC^{\infty}$ boundary and the damping acts on the entire boundary. In that work there is an additional time delay in the boundary condition and a certain nonlinearity is allowed for $k$.

  \item In \cite{PoSchn20} $\Omega$ is strongly star-shaped and has a $\conC^{5}$ boundary and the damping acts on the entire boundary.
        They allow $\epsilon$ and $\mu$ to be even state-dependent, i.e., quasilinear Maxwell's equations.

\end{itemize}

Note that the entire boundary $\partial\Omega$ always satisfies the geometric control condition. Hence, all of the above references work, at least implicitly, with this condition. Apart from \cite{Ph00} none of these references work with split/mixed boundary conditions.

There are also other effects that can stabilize Maxwell's equations like distributed damping or memory terms, see e.g., \cite{NiPi05,DoIoWa23}.

We will regard a domain $\Omega$ with Lipschitz boundary $\partial\Omega$ that is split into $\Gamma_{0}$ and $\Gamma_{1}$. The boundary damping acts only on one part, namely $\Gamma_{1}$---note that $\Gamma_{1}$ does not need to be connected. Moreover, $\epsilon$ and $\mu$ are Lipschitz continuous positive matrix-valued functions that are uniformly bounded from above and below.
However, we do not show exponential stability, but semi-uniform stability and we do not prove an explicit decay rate. Most likely such a decay rate will depend on the geometry of $\Omega$ and $\Gamma_{1}$.

However, in contrast to the listed literature we can, for example, deal with anisotropic Maxwell's equations on a cube that is damped on one face of the cube.



We will follow a similar approach as in~\cite{JaSk21}, where semi-uniform stability for the wave equations was shown. That is splitting the equations in a time independent (for equilibriums) and a time dependent (for the ``actual'' dynamic) part and showing that the differential operator that corresponds to the time dependent part has no spectrum on the imaginary axis. This will be done by showing that the operator has a compact resolvent and employing a unique continuation principle.

The main theorem of this work reads as follows.

\begin{restatable}{theorem}{maintheorem}\label{th:main-theorem}
  Let
  \(
  \begin{bsmallmatrix}
    \vec{E}_{0} \\ \vec{H}_{0}
  \end{bsmallmatrix}
  \in
  \begin{bsmallmatrix}
    \hH(\rot,\Omega) \cap \epsilon^{-1} \Hspace(\div,\Omega) \\
    \hH_{\Gamma_{1}}(\rot,\Omega) \cap \mu^{-1} \hH_{\Gamma_{0}}(\div,\Omega)
  \end{bsmallmatrix}
  \)
  be an initial value that satisfies the boundary conditions
  \begin{align*}
    \tantr \vec{E}_{0} = 0, \normaltr \vec{H}_{0}= 0 \ \text{on}\ \Gamma_{0}
    \quad\text{and}\quad
    \tantr \vec{E} + k \tanxtr \vec{H} = 0\ \text{on} \ \Gamma_{1},
  \end{align*}
  and Gau{\ss} laws $\div \epsilon \vec{E}_{0} = \rho$ and $\div \mu \vec{H}_{0} = 0$.
  Then the corresponding solution
  \(
  \begin{bsmallmatrix}
    \vec{E} \\ \vec{H}
  \end{bsmallmatrix}
  (t,\zeta)
  \)
  of \eqref{eq:maxwells-eq} converges to an equilibrium state
  \(
  \begin{bsmallmatrix}
    \vec{E}_{\mathrm{e}} \\ \vec{H}_{\mathrm{e}}
  \end{bsmallmatrix}
  (\zeta)
  \)
  for $t \to \infty$.
  More precisely there exists a monotone decreasing $f \colon [0,+\infty) \to [0,+\infty)$ with $\lim_{t \to +\infty} f(t) = 0$, which is independent of the initial values such that
  \begin{align*}
    \norm*{\begin{bmatrix} \vec{E} \\ \vec{H}\end{bmatrix}(t,\cdot)
    - \begin{bmatrix} \vec{E}_{\mathrm{e}} \\ \vec{H}_{\mathrm{e}}\end{bmatrix}
    }_{\Lp{2}(\Omega)}
    \leq f(t) \left(
    \norm*{\begin{bmatrix} \vec{E}_{0} \\ \vec{H}_{0}\end{bmatrix} -
    \begin{bmatrix} \vec{E}_{\mathrm{e}} \\ \vec{H}_{\mathrm{e}}\end{bmatrix}
    }_{\Lp{2}(\Omega)}
    + \norm*{\begin{bmatrix} \rot \vec{E}_{0} \\ \rot \vec{H}_{0}\end{bmatrix}}_{\Lp{2}(\Omega)}
    \right).
  \end{align*}
\end{restatable}

\begin{remark}
  Clearly, we can replace the $\Lp{2}$ norm in the previous theorem by a weighted $\Lp{2}$ norm, e.g., the energy norm that is induced by $\epsilon$ and $\mu$.
\end{remark}

\begin{remark}
  Note that the connectedness of $\Omega$ is not really necessary as long as $\Gamma_{1}$ has parts on every connected component.
\end{remark}

\begin{remark}
  We can actually also allow inhomogeneous boundary conditions for $\tantr \vec{E}$ and $\normaltr \vec{H}$ on $\Gamma_{0}$, because they will disappear in the equilibrium. We only have to make sure that the inhomogeneity satisfies certain compatibility conditions (not every $\Lp{2}$ function is in the range of $\tantr$).
\end{remark}




This work is structured in the following way: We will start by recalling the Sobolev spaces that correspond to our differential operators in \Cref{sec:preliminary}. Then we will split the system into a static and a dynamic part in \Cref{sec:split-the-system}. In \Cref{sec:semi-uniform-stability} we show that the dynamic part is semi-uniformly stable by a compact embedding and a unique continuation principle, which finally implies the main result \Cref{th:main-theorem}.

\section{Preliminary}\label{sec:preliminary}
For $\Omega \subseteq \R^{d}$ open and $\Gamma \subseteq \partial\Omega$ open we use the following notation (as in \cite{BaPaScho16})
\begin{align*}
  \Cc(\Omega) &\coloneqq \dset[\big]{f \in \conC^{\infty}(\R^{d})}{\supp f \subseteq \Omega \ \text{is compact}} \\
  \Cc_{\Gamma}(\Omega) &\coloneqq \dset[\Big]{f \big\vert_{\Omega}}{f \in \Cc(\R^{d}), \dist(\Gamma,\supp f) > 0}.
\end{align*}
We will regard an open, bounded and connected $\Omega \subseteq \R^{3}$ with strongly Lipschitz boundary.
For $g \in \Cc(\R^{3})$ and $f \in \Cc(\R^{3})^{3}$ we define the differential operators
\begin{align*}
  \grad g =
  \begin{bmatrix}
    \partial_{1}g \\ \partial_{2}g \\ \partial_{3}g
  \end{bmatrix},
  \quad
  \div f = \partial_{1} f_{1} + \partial_{2} f_{2} + \partial_{3} f_{3}
  \quad\text{and}\quad
  \rot f =
  \begin{bmatrix}
    \partial_{2} f_{3} - \partial_{3} f_{2} \\
    \partial_{3} f_{1} - \partial_{1} f_{3} \\
    \partial_{1} f_{2} - \partial_{2} f_{1}
  \end{bmatrix}.
\end{align*}
These operators can be regarded as operators from $\Lp{2}(\Omega)^{k_{1}}$ to $\Lp{2}(\Omega)^{k_{2}}$ (with $k_{1},k_{2} \in \N$ suitable). In the further we will omit the exponent $k$ at $\Lp{2}(\Omega)^{k}$, $\Hspace^{1}(\Omega)^{k}$, $\Cc(\Omega)^{k}$, etc., if they are clear from the context.
We introduce the maximal domain of these operators on $\Lp{2}(\Omega)$:
\begin{align*}
  \Hspace^{1}(\Omega) &= \dset{g \in \Lp{2}(\Omega)}{\grad g \in \Lp{2}(\Omega)^{3}}, \\
  \Hspace(\div,\Omega) &= \dset{f \in \Lp{2}(\Omega)^{3}}{\div f \in \Lp{2}(\Omega)}, \\
  \mathllap{\text{and}\quad}\Hspace(\rot,\Omega) &= \dset{f \in \Lp{2}(\Omega)^{3}}{\rot f \in \Lp{2}(\Omega)^{3}}, \mathrlap{\quad\text{respectively.}}
\end{align*}
For $g \in \Cc(\R^{3})$ and $f \in \Cc(\R^{3})^{3}$ there is the well-known integration by parts formula
\begin{align*}
  \scprod{\div f}{g}_{\Lp{2}(\Omega)} + \scprod{f}{\grad g}_{\Lp{2}(\Omega)} = \scprod*{\nu \cdot f\big\vert_{\partial\Omega}}{g\big\vert_{\partial\Omega}}_{\Lp{2}(\partial\Omega)},
\end{align*}
where $\nu$ denotes the outward pointing unit normal vector on the boundary.
This formula can be extended to the maximal domain of the respective differential operator, such that we have for $g \in \Hspace^{1}(\Omega)$ and $f \in \Hspace(\div,\Omega)$
\begin{align*}
  \scprod{\div f}{g}_{\Lp{2}(\Omega)} + \scprod{f}{\grad g}_{\Lp{2}(\Omega)} = \dualprod*{\normaltr f}{\boundtr g}_{\Hspace^{-\nicefrac{1}{2}}(\partial\Omega),\Hspace^{\nicefrac{1}{2}}(\partial\Omega)},
\end{align*}
where $\boundtr$, the boundary trace, is the extension of $g \mapsto g\big\vert_{\partial\Omega}$ to $\Hspace^{1}(\Omega)$ and $\normaltr$, the normal trace, is the extension of $f \mapsto \nu \cdot f\big\vert_{\partial\Omega}$ to $\Hspace(\div,\Omega)$. These mappings map onto $\Hspace^{\nicefrac{1}{2}}(\partial\Omega)$ and $\Hspace^{-\nicefrac{1}{2}}(\partial\Omega)$ respectively, where $\Hspace^{\nicefrac{1}{2}}(\partial\Omega)$ is range of $\boundtr$ endowed with the range norm of $\Hspace^{1}(\Omega)$, and $\Hspace^{-\nicefrac{1}{2}}(\partial\Omega)$ is its dual space.

Basically, for $\rot$ there is a similar approach to extend the integration by parts formula
\begin{align*}
  \scprod{\rot f}{g}_{\Lp{2}(\Omega)} + \scprod{f}{-\rot g}_{\Lp{2}(\Omega)} = \scprod{\nu \times f}{(\nu \times g) \times \nu}_{\Lp{2}(\partial\Omega)}
\end{align*}
that is valid for $f, g \in \Cc(\R^{3})$ to the maximal domain of $\rot$. We present this in \Cref{sec:background-diffoperator-and-traces}.

Note that every $w \in \C^{3}$ can be represented as
\begin{equation*}
  w = \underbrace{(\nu(\zeta) \times w) \times \nu(\zeta)}_{\text{tangential part}}\mathclose{} + \mathopen{}\underbrace{(\nu(\zeta) \cdot w) \nu(\zeta)}_{\text{normal part}}
  \quad\text{for a.e.}\quad \zeta \in \partial\Omega.
\end{equation*}
Hence, we call $\tantr$, the extension of $g \mapsto (\nu \times g \big\vert_{\partial\Omega}) \times \nu$, the \emph{tangential trace} and $\tanxtr$, the extension of $f \mapsto \nu \times g \big\vert_{\partial\Omega}$, the \emph{twisted tangential trace}. If we want to emphasize that we are only interested on the part $\Gamma_{1} \subseteq \partial\Omega$ of the tangential trace we denote this by $\tantr[\Gamma_{1}]$ and $\tanxtr[\Gamma_{1}]$ for details see \Cref{sec:background-diffoperator-and-traces}.


For two Hilbert space $X$ and $Y$ we will use the notation
\begin{align*}
  \begin{bmatrix}
    X \\ Y
  \end{bmatrix}
  \coloneqq
  X \times Y.
\end{align*}
For a strictly positive and bounded operator $P$ on $\Lp{2}(\Omega)$ we introduce
\begin{align}
  \label{eq:def-inner-product-with-operator}
  \scprod{x}{y}_{P} \coloneqq \scprod{x}{Py}_{\Lp{2}(\Omega)} = \scprod{Px}{y}_{\Lp{2}(\Omega)},
\end{align}
which establishes an equivalent inner product on $\Lp{2}(\Omega)$. Corresponding to this new inner product we have $\perp_{P}$ and $\oplus_{P}$.

For a strongly Lipschitz domain $\Omega$ and $\Gamma_{0} \subseteq \partial\Omega$ we say that the pair $(\Omega,\Gamma_{0})$ is a strongly Lipschitz pair, roughly speaking, if also $\partial \Gamma_{0}$ is strongly Lipschitz. For details see \cite{BaPaScho16,SkPa23}.



Additionally to the restricted tangential traces ($\tantr[\Gamma_{1}]$ and $\tanxtr[\Gamma_{1}]$), we introduce the extension of $g \mapsto \nu \cdot g \big\vert_{\Gamma}$ to $\Hspace(\div,\Omega)$ denoted by $\normaltr[\Gamma] g$.

Furthermore, we define
\begin{align*}
  \cH_{\Gamma}(\rot, \Omega) &\coloneqq \dset{f \in \Lp{2}(\Omega)}{\rot f \in \Lp{2}(\Omega), \tantr[\Gamma] f = 0} \\
  \hH_{\Gamma}(\rot, \Omega) &\coloneqq \dset{f \in \Lp{2}(\Omega)}{\rot f \in \Lp{2}(\Omega), \tantr[\Gamma] f \in \Lp{2}(\Gamma)}
\end{align*}
and
\begin{align*}
  \cH_{\Gamma}(\div, \Omega) &\coloneqq \dset{f \in \Lp{2}(\Omega)}{\div f \in \Lp{2}(\Omega), \normaltr[\Gamma] f = 0} \\
  \hH_{\Gamma}(\div, \Omega) &\coloneqq \dset{f \in \Lp{2}(\Omega)}{\div f \in \Lp{2}(\Omega), \normaltr[\Gamma] f \in \Lp{2}(\Gamma)},
\end{align*}
see e.g., \cite{PaSk21}.
Similar to \cite{DaLi90} we denote the spaces with vanishing $\rot$ and $\div$, respectively, by
\begin{align*}
  \renewcommand{\quad}{\mspace{30mu}}
  \begin{aligned}[t]
    \Hspace(\rot 0,\Omega) &\coloneqq \ker \rot, \\
    \cH_{\Gamma}(\rot 0, \Omega) &\coloneqq \cH_{\Gamma}(\rot, \Omega) \cap \ker \rot, \\
    \hH_{\Gamma}(\rot 0, \Omega) &\coloneqq \hH_{\Gamma}(\rot, \Omega) \cap \ker \rot,
  \end{aligned}
                                   \quad
                                   \begin{aligned}[t]
                                     \Hspace(\div 0,\Omega) &\coloneqq \ker \div, \\
                                     \cH_{\Gamma}(\div 0, \Omega) &\coloneqq \cH_{\Gamma}(\div, \Omega) \cap \ker \div, \\
                                     \hH_{\Gamma}(\div 0, \Omega) &\coloneqq \hH_{\Gamma}(\div, \Omega) \cap \ker \div.
                                   \end{aligned}
\end{align*}
Moreover, we define the cohomology groups for $\delta \in \Lb(\Lp{2}(\Omega))$ strictly positive by
\begin{align*}
  \cohom_{\Gamma_{a},\Gamma_{b},\delta}(\Omega)
  &\coloneqq \cH_{\Gamma_{a}}(\rot 0, \Omega) \cap \delta^{-1} \cH_{\Gamma_{b}}(\div 0,\Omega)
  \\
  \cohom_{\delta,\Gamma_{a},\Gamma_{b}}(\Omega)
  & \coloneqq \delta \cohom_{\Gamma_{a},\Gamma_{b},\delta}(\Omega)
  \\
  &\phantom{\mathclose{} \coloneqq \mathopen{}}\mathllap{\mathclose{}=\mathopen{}} \delta \cH_{\Gamma_{a}}(\rot 0, \Omega) \cap \cH_{\Gamma_{b}}(\div 0,\Omega).
\end{align*}

\section{Split the system}\label{sec:split-the-system}
In this section we will split our system in a time invariant part and the remaining dynamic part. This will simplify the analysis of the spectrum of the (remaining) dynamic part. From a certain point of view we factor out the eigenvectors corresponding to the eigenvalue $0$ of the differential operator that describes the dynamic.

We can write the system~\eqref{eq:maxwells-eq-dynamic-1}-\eqref{eq:maxwells-eq-dynamic-2} as
\begin{align*}
  \frac{\partial}{\partial t}
  \begin{bmatrix}
    \vec{D} \\ \vec{B}
  \end{bmatrix}
  =
  \begin{bmatrix}
    0 & \rot \\
    - \rot & 0
  \end{bmatrix}
             \begin{bmatrix}
               \epsilon^{-1} & 0 \\
               0 & \mu^{-1}
             \end{bmatrix}
                   \begin{bmatrix}
                     \vec{D} \\ \vec{B}
                   \end{bmatrix}
\end{align*}
or as
\begin{align*}
  \frac{\partial}{\partial t}
  \begin{bmatrix}
    \epsilon & 0 \\
    0 & \mu
  \end{bmatrix}
  \begin{bmatrix}
    \vec{E} \\ \vec{H}
  \end{bmatrix}
  =
  \begin{bmatrix}
    0 & \rot \\
    - \rot & 0
  \end{bmatrix}
             \begin{bmatrix}
               \vec{E} \\ \vec{H}
             \end{bmatrix}
\end{align*}
From a semigroup perspective the first version is in a better form.
The second version is for instance favored by the approach in \cite{SeTrWa22}.

In order to analyze the stability of the system we have to separate the equilibrium states from the rest of the system.
We call the remaining system the ``true dynamic''.

By setting all time derivatives to zero in~\eqref{eq:maxwells-eq}, we obtain the equations for the equilibrium states.
\begin{subequations}
  \label{eq:maxwell-static}
  \begin{align}
    \rot \vec{E} &= 0, & \rot \vec{H} &= 0, \\
    \div \epsilon \vec{E} &= \rho, & \div \mu \vec{H} &= 0, \\
    \tantr[\Gamma_{0}] \vec{E} &= 0, & \normaltr[\Gamma_{0}] \mu \vec{H} &= 0, \\
    \tantr[\Gamma_{1}] \vec{E} &= -k h, & \tanxtr[\Gamma_{1}] \vec{H} &= h,
  \end{align}
\end{subequations}
where $h$ is determined by the traces of the initial values. This static system is solvable by~\cite[Thm.~5.6]{BaPaScho16}.
Note that if $\Omega$ is not simply connected, then the cohomology spaces
\begin{align*}
  \cohom_{\partial\Omega,\emptyset,\epsilon}(\Omega) \quad\text{and}\quad \cohom_{\Gamma_{1},\Gamma_{0},\mu}(\Omega)
\end{align*}
contain more than the zero vector. Hence, the solutions of \eqref{eq:maxwell-static} are not unique,
because for a solution
\(
\begin{bsmallmatrix}
  \vec{E}_{\mathrm{e}} \\ \vec{H}_{\mathrm{e}}
\end{bsmallmatrix}
\)
also
\(
\begin{bsmallmatrix}
  \vec{E}_{\mathrm{e}} \\ \vec{H}_{\mathrm{e}}
\end{bsmallmatrix}
+
\begin{bsmallmatrix}
  \tilde{\vec{E}} \\ \tilde{\vec{H}}
\end{bsmallmatrix}
\)
solves \eqref{eq:maxwell-static}, if
\(
\begin{bsmallmatrix}
  \tilde{\vec{E}} \\ \tilde{\vec{H}}
\end{bsmallmatrix}
\in
\begin{bsmallmatrix}
  \cohom_{\partial\Omega,\emptyset,\epsilon}(\Omega) \\ \cohom_{\Gamma_{1},\Gamma_{0},\mu}(\Omega)
\end{bsmallmatrix}
\).

For the ``true'' dynamic we regard the operator
\[
\begin{bmatrix}
  0 & \rot \\
  -\rot & 0
\end{bmatrix}
\begin{bmatrix}
  \epsilon^{-1} & 0 \\
  0 & \mu^{-1}
\end{bmatrix}
\]
with the boundary conditions
\[
  \tantr[\Gamma_{0}] \epsilon^{-1} \vec{D} = 0 \quad\text{and}\quad \tantr[\Gamma_{1}] \epsilon^{-1} \vec{D} + k \tanxtr[\Gamma_{1}] \mu^{-1} \vec{B} = 0.
\]
Note that in general $\tantr[\Gamma_{1}]$ and $\tanxtr[\Gamma_{1}]$ map into different spaces. Hence, in order to meaningfully regard the second boundary condition we restrict ourselves to all those elements that are mapped into the pivot space $\Lp{2}_{\tau}(\Gamma_{1})$ under $\tantr[\Gamma_{1}]$, $\tanxtr[\Gamma_{1}]$.\footnote{One can also use quasi Gelfand triple theory to regard the boundary condition in a larger space (that contains $\Vtau(\Gamma_{1})$ and $\Vtaud(\Gamma_{1})$), but in the end this only implies that all vector fields $\vec{D}$, $\vec{B}$ whose tangential traces satisfy this condition in the larger space are already in the pivot space $\Lp{2}_{\tau}(\Gamma_{1})$.}
Combined with the boundary condition $\tantr[\Gamma_{0}] \epsilon^{-1} \vec{D} = 0$ we can say that $\epsilon^{-1} \vec{D}$ has an $\Lp{2}$ tangential trace on the entire boundary.
Hence, in order to satisfy our boundary conditions we require/assume that
\[
\vec{E} = \epsilon^{-1}\vec{D} \in \hH_{\partial\Omega}(\rot,\Omega) \cap \cH_{\Gamma_{0}}(\rot,\Omega) \quad\text{and}\quad \vec{H} = \mu^{-1}\vec{B} \in \hH_{\Gamma_{1}}(\rot,\Omega).
\]
Summarized we regard the following operator and domain
\begingroup 
\thinmuskip=2mu%
\medmuskip=3mu%
\thickmuskip=4mu plus 2mu%
\begin{align}
  \label{eq:def-A-0}
  \begin{split}
    A_{0}
    &=
      \begin{bmatrix}
        0 & \rot \\
        -\rot & 0
      \end{bmatrix}
                \overbrace{
                \begin{bmatrix}
                  \epsilon^{-1} & 0 \\
                  0 & \mu^{-1}
                \end{bmatrix}}^{\eqqcolon\mathrlap{\hamiltonian}}
    \\[0.5em]
    \dom A_{0}
    &= \left\{
      \begin{bmatrix}
        \vec{D} \\ \vec{B}
      \end{bmatrix}
    \in
    \hamiltonian^{-1}
    \begin{bmatrix}
      \hH_{\partial\Omega}(\rot,\Omega) \cap \cH_{\Gamma_{0}}(\rot,\Omega)\\ \hH_{\Gamma_{1}}(\rot,\Omega)
    \end{bmatrix}
      \right. \\
    &\pushright{
      \,\left|\,
      \vphantom{\begin{bmatrix} \vec{D} \\ \vec{B}\end{bmatrix}}
      \tantr[\Gamma_{1}] \epsilon^{-1} \vec{D} + k \tanxtr[\Gamma_{1}] \mu^{-1} \vec{B} = 0
      \right\}\mathrlap{\phantom{.}}
      }
    \\[0.5em]
    &= \hamiltonian^{-1}
      \dset*{
      \begin{bmatrix}
        \vec{E} \\ \vec{H}
      \end{bmatrix}
    \in
    \begin{bmatrix}
      \hH_{\partial\Omega}(\rot,\Omega) \cap \cH_{\Gamma_{0}}(\rot,\Omega)\\ \hH_{\Gamma_{1}}(\rot,\Omega)
    \end{bmatrix}
    }{\tantr[\Gamma_{1}] \vec{E} + k \tanxtr[\Gamma_{1}] \vec{H} = 0}\mathrlap{.}
  \end{split}
\end{align}
\endgroup
The operator $A_{0}$ is a generator of a contraction semigroup as we will explain in \Cref{sec:generator}. 
Note that $\hamiltonian \colon \Lp{2}(\Omega) \to \Lp{2}(\Omega)$ is a strictly positive and bounded operator, by the assumptions on $\epsilon$ and $\mu$.
It comes very natural to use $\scprod{\cdot}{\cdot}_{\hamiltonian}$ as an inner product, since the corresponding norm is the energy norm in the state space.

In order to satisfy the remaining conditions we define the state space
\begin{align*}
  \XH \coloneqq \dset*{
  \begin{bsmallmatrix}
    \vec{D} \\ \vec{B}
  \end{bsmallmatrix}
  }{\div \vec{D} = 0, \div \vec{B} = 0, \normaltr[\Gamma_{0}] \vec{B} = 0}
  \cap
  \begin{bmatrix}
    \cohom_{\epsilon,\partial\Omega,\emptyset}(\Omega) \\ \cohom_{\mu,\Gamma_{1},\Gamma_{0}}(\Omega)
  \end{bmatrix}^{\perp_{\hamiltonian}}.
\end{align*}
Clearly, $\XH \subseteq \Lp{2}(\Omega)$. Note that
\begin{align*}
  \begin{bmatrix}
    \cohom_{\epsilon,\partial\Omega,\emptyset}(\Omega) \\ \cohom_{\mu,\Gamma_{1},\Gamma_{0}}(\Omega)
  \end{bmatrix}^{\perp_{\hamiltonian}}
  =
  \begin{bmatrix}
    \cohom_{\epsilon,\partial\Omega,\emptyset}(\Omega)^{\perp_{\epsilon^{-1}}} \\ \cohom_{\mu,\Gamma_{1},\Gamma_{0}}(\Omega)^{\perp_{\mu^{-1}}}
  \end{bmatrix}.
\end{align*}
The intersection with $\big(\cohom_{\epsilon,\partial\Omega,\emptyset}(\Omega) \times \cohom_{\mu,\Gamma_{1},\Gamma_{0}}(\Omega)\big)^{\perp_{\hamiltonian}}$ factors out all solutions in $\cohom_{\epsilon,\partial\Omega,\emptyset}(\Omega) \times \cohom_{\mu,\Gamma_{1},\Gamma_{0}}(\Omega)$, because these are static solutions and already included in the static solutions (as difference of two solutions of \eqref{eq:maxwell-static}).

\begin{lemma}\label{le:decomposition-of-Hilbert-space-deformed-by-postive-operatore}
  Let $H$ be a Hilbert space that can be decomposed into $H = U_{1} \oplus U_{2}$, where $U_{1}$ and $U_{2}$ are closed subspaces of $H$. Then for every strictly positive operator $Q \in \Lb(H)$ we can decompose $H$ into
  \begin{align*}
    H = U_{1} \oplus U_{2} = U_{1} \oplus_{Q} Q^{-1} U_{2} = U_{1} \oplus_{Q^{-1}} Q U_{2},
  \end{align*}
  where $\oplus_{Q}$ denotes the orthogonal sum w.r.t.\ the inner product $\scprod{x}{y}_{Q} \coloneqq \scprod{x}{Qy}$.
\end{lemma}

\begin{proof}
  It is straightforward to show that $U_{1} \perp_{Q} Q^{-1}U_{2}$. Hence, it is left to show that $H$ is the direct sum of $U_{1}$ and $Q^{-1}U_{2}$. Note that $U_{1}$ and $Q^{-1}U_{2}$ are both closed. Therefore, the orthogonal projections (w.r.t.\ $\scprod{\cdot}{\cdot}_{Q}$) $P_{1}$ and $P_{2}$ on $U_{1}$ and $Q^{-1}U_{2}$, respectively are continuous and $U_{1} \oplus_{Q} Q^{-1}U_{2}$ is closed.

  Let us assume that $U_{1} \oplus_{Q} Q^{-1}U_{2}$ does not already equal $H$. Then there exists an $h \in H$ such that
  $h \perp w_{1} + Q^{-1}w_{2}$ for all $w_{1} \in U_{1}$ and $w_{2} \in U_{2}$. Note that by assumption $h = u_{1} + u_{2}$ for $u_{1} \in U_{1}$ and $u_{2} \in U_{2}$. Hence, for $w_{1} = u_{1}$ and $w_{2} = 0$ we obtain
  \begin{align*}
    0 = \scprod{h}{u_{1}} = \norm{u_{1}}^{2}
  \end{align*}
  and therefore $h = u_{2}$. For $w_{1} = 0$ and $w_{2} = u_{2}$ we have
  \begin{align*}
    0 = \scprod{h}{Q^{-1} u_{2}} = \scprod{u_{2}}{Q^{-1} u_{2}} \geq C \norm{u_{2}}^{2},
  \end{align*}
  which implies $h = 0$.
\end{proof}

The next theorem is \cite[Thm.~5.3]{BaPaScho16}, we have just made some minor modifications to better suit our setting.

\begin{restatable}[Helmholtz decomposition]{theorem}{helmholtzdecomposition}\label{th:helmholtz-decomposition}
  Let $(\Omega, \Gamma_{a})$ be a strongly Lipschitz pair and $\Gamma_{b} = \partial\Omega \setminus \cl{\Gamma_{a}}$, and $\delta \in \Lb(\Lp{2}(\Omega))$ a strictly positive operator. Then
  \begin{align*}
    \Lp{2}(\Omega) =
    \delta \grad \cH^{1}_{\Gamma_{a}}(\Omega)
    \oplus_{\delta^{-1}}
    \cohom_{\delta,\Gamma_{a},\Gamma_{b}}(\Omega)
    \oplus_{\delta^{-1}}
    \rot \cH_{\Gamma_{b}}(\rot,\Omega)
  \end{align*}
\end{restatable}

\begin{proof}
  Note that the adjoint of $D \coloneqq \grad$ as operator on $\Lp{2}(\Omega)$ with $\dom D = \cH^{1}_{\Gamma_{a}}(\Omega)$ is $-\div$ with domain $\cH_{\Gamma_{b}}(\div,\Omega)$.
  Hence,
  \begin{align}\label{eq:L2-as-ran-grad-plus-ker-div}
    \begin{split}
      \Lp{2}(\Omega)
      &= \cl{\ran D} \oplus \ker D\adjun \\
      &= \grad \cH^{1}_{\Gamma_{a}}(\Omega) \oplus \cH_{\Gamma_{b}}(\div 0,\Omega) \\
      &\stackrel{\mathclap{\text{
        \hyperref[le:decomposition-of-Hilbert-space-deformed-by-postive-operatore]
        {L.\ref{le:decomposition-of-Hilbert-space-deformed-by-postive-operatore}\;\;}
        }}}{=}
        \delta \grad \cH^{1}_{\Gamma_{a}}(\Omega) \oplus_{\delta^{-1}} \cH_{\Gamma_{b}}(\div 0,\Omega)
    \end{split}
  \end{align}
  where the closedness of $\grad \cH^{1}_{\Gamma_{a}}(\Omega)$ is a consequence of Poincar{\'e}'s inequality.
  Since  $\grad \cH^{1}_{\Gamma_{a}}(\Omega) \subseteq \cH_{\Gamma_{a}}(\rot 0, \Omega)$, the intersection of \eqref{eq:L2-as-ran-grad-plus-ker-div} and $\delta \cH_{\Gamma_{a}}(\rot 0,\Omega)$ leads to
  \begin{equation}
    \label{eq:ker-rot-identity}
    \delta \cH_{\Gamma_{a}}(\rot 0, \Omega)
    =
    \delta \grad \cH^{1}_{\Gamma_{a}}(\Omega)
    \oplus_{\delta^{-1}}
    \underbrace{
    \delta \cH_{\Gamma_{a}}(\rot 0, \Omega) \cap \cH_{\Gamma_{b}}(\div 0,\Omega)
    }_{=\mathrlap{\cohom_{\delta,\Gamma_{a},\Gamma_{b}}(\Omega)}}
  \end{equation}
  By the same argument for $D \coloneqq \rot$ with $\dom D = \cH_{\Gamma_{b}}(\rot,\Omega)$ we have
  \begin{align}
    \label{eq:L2-as-ran-rot-plus-ker-rot}
    \begin{split}
      \Lp{2}(\Omega)
      &= {\rot \cH_{\Gamma_{b}}(\rot,\Omega)} \oplus \cH_{\Gamma_{a}}(\rot 0,\Omega) \\
      &= {\rot \cH_{\Gamma_{b}}(\rot,\Omega)} \oplus_{\delta^{-1}} \delta \cH_{\Gamma_{a}}(\rot 0,\Omega),
    \end{split}
  \end{align}
  where the closedness of $\rot \cH_{\Gamma_{b}}(\rot,\Omega)$ follows from~\cite[Lem.~5.2]{BaPaScho16}.
  Combining~\eqref{eq:L2-as-ran-rot-plus-ker-rot} and~\eqref{eq:ker-rot-identity} leads to
  \begin{equation*}
    \Lp{2}(\Omega) =
    \delta \grad \cH^{1}_{\Gamma_{a}}(\Omega)
    \oplus_{\delta^{-1}}
    \cohom_{\delta,\Gamma_{a},\Gamma_{b}}(\Omega)
    \oplus_{\delta^{-1}}
    {\rot \cH_{\Gamma_{b}}(\rot,\Omega)}.
    \qedhere
  \end{equation*}
\end{proof}

\begin{corollary}\label{th:div0-cap-cohomology}
  Let $(\Omega,\Gamma_{a})$ be a strongly Lipschitz pair, $\Gamma_{b} = \partial\Omega \setminus \cl{\Gamma_{a}}$ and $\delta \in \Lb(\Lp{2}(\Omega))$ a strictly positive operator. Then
  \begin{align*}
    \cH_{\Gamma_{b}}(\div 0, \Omega) \cap (\cohom_{\delta,\Gamma_{a},\Gamma_{b}}(\Omega))^{\perp_{\delta^{-1}}} = \rot \cH_{\Gamma_{b}}(\rot,\Omega).
  \end{align*}
\end{corollary}

\begin{proof}
  By \Cref{th:helmholtz-decomposition} and~\eqref{eq:L2-as-ran-grad-plus-ker-div} we have
  \begin{multline*}
    \delta \grad \cH^{1}_{\Gamma_{a}}(\Omega)
    \oplus_{\delta^{-1}}
    \cohom_{\delta,\Gamma_{a},\Gamma_{b}}(\Omega)
    \oplus_{\delta^{-1}}
    \rot \cH_{\Gamma_{b}}(\rot,\Omega)
    \\
    = \Lp{2}(\Omega) =
    \delta \grad \cH^{1}_{\Gamma_{a}}(\Omega) \oplus_{\delta^{-1}} \cH_{\Gamma_{b}}(\div 0,\Omega).
  \end{multline*}
  Hence,
  \begin{equation*}
    \cohom_{\delta,\Gamma_{a},\Gamma_{b}}(\Omega)
    \oplus_{\delta^{-1}}
    \rot \cH_{\Gamma_{b}}(\rot,\Omega)
    =
    \cH_{\Gamma_{b}}(\div 0,\Omega)
  \end{equation*}
  and intersecting both sides with $(\cohom_{\delta,\Gamma_{a},\Gamma_{b}}(\Omega))^{\perp_{\delta^{-1}}}$ finishes the proof.
\end{proof}

\begin{lemma}
  The space $\XH$ is a Hilbert space with the inner product $\scprod{x}{y}_{\XH} \coloneqq \scprod{x}{y}_{\hamiltonian} = \scprod{x}{\hamiltonian y}_{\Lp{2}(\Omega)}$.
  Moreover, $\ran A_{0} \subseteq \XH$ and in particular $\XH$ is invariant under the semigroup $T_{0}$ that is generated by $A_{0}$.
\end{lemma}

\begin{proof}
  Since $\hamiltonian \colon \Lp{2}(\Omega) \to \Lp{2}(\Omega)$ is a positive, bounded and boundedly invertible operator, $\scprod{\cdot}{\cdot}_{\Lp{2}(\Omega)}$ is equivalent to $\scprod{\cdot}{\cdot}_{\XH}$. The closedness of the operators $\div$ and $\normaltr[\Gamma_{0}]$ implies the closedness of $\ker \div$ and $\ker \normaltr[\Gamma_{0}]$. Hence, $\XH$ is the intersection of closed subspaces and therefore closed.

  By the definition of $\XH$ and \Cref{th:div0-cap-cohomology} we have
  \begin{align*}
    \XH
    &=
    \begin{bmatrix}
      \Hspace(\div 0,\Omega) \cap \cohom_{\epsilon,\partial\Omega,\emptyset}(\Omega)^{\perp_{\epsilon^{-1}}} \\
      \cH_{\Gamma_{0}}(\div,\Omega) \cap \cohom_{\mu,\Gamma_{1},\Gamma_{0}}(\Omega)^{\perp_{\mu^{-1}}}
    \end{bmatrix}
    =
    \begin{bmatrix}
      \rot \Hspace(\rot,\Omega) \\ \rot \cH_{\Gamma_{0}}(\rot,\Omega)
    \end{bmatrix}
    \\[1ex]
    &=
      \begin{bmatrix}
        0 & \rot \\
        -\rot & 0
      \end{bmatrix}
                \begin{bmatrix}
                  \cH_{\Gamma_{0}}(\rot,\Omega) \\
                  \Hspace(\rot,\Omega)
                \end{bmatrix}
    =
      \smash[b]{%
      \underbrace{%
      \begin{bmatrix}
        0 & \rot \\
        -\rot & 0
      \end{bmatrix}
                \hamiltonian}_{\supseteq\mathrlap{A_{0}}} \underbrace{\hamiltonian^{-1}
                \begin{bmatrix}
                  \cH_{\Gamma_{0}}(\rot,\Omega) \\
                  \Hspace(\rot,\Omega)
                \end{bmatrix}}_{\supseteq\mathrlap{\dom A_{0}}}
      } \\[1.5ex]
    &\supseteq \ran A_{0}.
  \end{align*}
  By~\cite[ch.\ II sec.\ 2.3]{engel-nagel}, $\XH \supseteq \ran A_{0}$ implies that $\XH$ is invariant under the semigroup $T_{0}$ that is generated by $A_{0}$.
\end{proof}

Finally, we introduce the differential operator and its domain that describes the dynamic of our system.
\begin{definition}
  \label{def:differential-operator-A}
  We define $A \coloneqq A_{0}\big\vert_{\XH}$, which we regard as an operator $A\colon \dom A \subseteq \XH \to \XH$. The domain of $A$ is given by $\dom A = \dom A_{0} \cap \XH$, i.e.,
  \begin{align*}
    A\colon\mapping{\dom A_{0} \cap \XH \subseteq \XH}{\XH}{x}{A_{0}x.}
  \end{align*}
\end{definition}

Note that $A$ is a generator of a contraction semigroup, since $A_{0}$ is a generator of a contraction semigroup and $\XH$ is closed and invariant under the semigroup $T_{0}$ generated by $A_{0}$. In particular, the semigroup $T$ that is generated by $A$ is given by $T(t) = T_{0}(t)\big\vert_{\XH}$, see e.g., \cite[ch. II sec. 2.3]{engel-nagel}.

\section{Semi-uniform stability}\label{sec:semi-uniform-stability}

We will regard the semigroup $T$ that is generated by the operator $A$ from the previous section, defined in \Cref{def:differential-operator-A}. Our goal is to show that this semigroup is semi-uniformly stable.

\begin{definition}\label{def:semi-uniform-stability}
  Let $A$ be the generator of the strongly continuous and bounded semigroup $(T(t))_{t\geq 0}$ on a Hilbert space $X$. Then we say that $(T(t))_{t\geq 0}$ is semi-uniformly stable, if there exists a continuous and decreasing function $f\colon [0,+\infty) \to [0,+\infty)$ with $\lim_{t\to+\infty} f(t) = 0$ and
  \begin{align}\label{eq:semi-uniform-stability}
    \norm{T(t)x}_{X} \leq f(t) \norm{x}_{\dom A}
  \end{align}
  for every $x \in \dom A$, where $\norm{\cdot}_{\dom A}$ denotes the graph norm of $A$.
\end{definition}

The difference between uniform stability and semi-uniform stability is that on the right-hand side of \eqref{eq:semi-uniform-stability} there is the graph norm of the generator $A$ instead of just the norm of the Hilbert space $X$.


It is very common to define semi-uniform stability by an equivalent characterization, namely item \textup{(iii)} of the next theorem.

\begin{theorem}[{\cite{BaDu08}}]\label{th:equivalences-of-semi-uniform-stability}
  Let $T$ be a strongly continuous and bounded semigroup generated by $A$.
  Then the following assertions are equivalent.
  \begin{enumerate}
    \item $T$ is semi-uniformly stable.
    \item $\norm{T(t)(A-\lambda)^{-1}} \to 0$ for all $\lambda \in \uprho(A)$.
    \item $\norm{T(t) A^{-1}} \to 0$.
    \item $\iu \R \subseteq \uprho(A)$.
  \end{enumerate}
\end{theorem}

Hence, by the previous theorem, the question about semi-uniform stability of $A$ reduces to the following problem.

\begin{problem}\label{pro:reduced-question}
  Show for every $\omega \in \R$ that $\iu \omega \in \uprho(A)$ or equivalently that $\iu \omega \notin \upsigma(A)$.
\end{problem}

A compact resolvent simplifies the task of finding resolvent points. Hence, the following theorem from \cite[Thm.~4.1]{PaSk21} will reduce the problem such that we only need make sure that $\iu \omega$ is not an eigenvalue, see \Cref{th:operator-with-compact-resolvent-only-eigenvalues}.

\begin{theorem}[Compact embedding, {\cite[Thm.~4.1]{PaSk21}}]
  \label{th:compact-embedding-for-div-rot-space}
  Let $(\Omega,\Gamma_{a})$ be a strongly Lipschitz pair, $\Gamma_{b} = \partial\Omega \setminus \cl{\Gamma_{a}}$ and $\delta\colon \Omega \to \R^{3\times 3}$ such that the corresponding multiplication operator is in $\Lb(\Lp{2}(\Omega))$ and strictly positive. Then
  \begin{equation*}
    \hH_{\Gamma_{a}}(\rot,\Omega) \cap \delta^{-1} \hH_{\Gamma_{b}}(\div,\Omega) \cpt \Lp{2}(\Omega),
  \end{equation*}
  where $\cpt$ denotes the existence of a compact embedding.
\end{theorem}

In particular, as a consequence of this compact embedding we obtain the following proposition. A special case for $\epsilon = \mu = 1$ has been regarded in~{\cite[Thm.~5.6]{PaSk21}}.
\begin{proposition}
  \label{th:A-has-compact-resolvent}
  The operator $A$ has a compact resolvent.
\end{proposition}

\begin{proof}
  By definition of $A$ we have $\dom A = \dom A_{0} \cap \XH$. Hence, by construction and \Cref{th:compact-embedding-for-div-rot-space} we have
  \begin{equation*}
    \dom A \subseteq
    \begin{bmatrix}
      \epsilon \hH_{\partial\Omega}(\rot,\Omega) \cap \Hspace(\div,\Omega) \\
      \mu \hH_{\Gamma_{1}}(\rot,\Omega) \cap \hH_{\Gamma_{0}}(\div,\Omega)
    \end{bmatrix}
    \cpt \Lp{2}(\Omega).
    \qedhere
  \end{equation*}
\end{proof}

Hence, $A$ has only point spectrum. So we only need to check, whether $\ker (A - \iu \omega) = \set{0}$ for all $\omega \in \R$. For convenience we provide a proof for this conclusion.

\begin{theorem}\label{th:operator-with-compact-resolvent-only-eigenvalues}
  Let $\mathcal{X}$ be a Hilbert space and $A\colon \dom A \subseteq \mathcal{X} \to \mathcal{X}$ be a closed and densely defined operator with compact resolvent. Then $\upsigma(A) = \upsigma_{\mathrm{p}}(A)$, i.e., the spectrum of $A$ contains only eigenvalues. Moreover, $\upsigma(A)$ is countable and has no accumulation points (or at most $\infty$ as accumulation point, if we regard the spectrum as subset of $\C \cup \set{\infty}$).
\end{theorem}

\begin{proof}
  Let $\lambda \in \uprho(A)$. Then $(A-\lambda)^{-1}$ is a compact operator. Hence, the spectrum of $(A-\lambda)^{-1}$ contains only countable eigenvalues that can only accumulate at $0$. By the spectral mapping theorem (for unbounded operators or linear relations) we have
  \begin{equation*}
    \upsigma\big((A-\lambda)^{-1}\big) = \big(\upsigma(A-\lambda)\big)^{-1}.
  \end{equation*}
  Hence, the claim is true for $A-\lambda$. By shifting this operator by $\lambda$ and again apply the spectral mapping theorem, we conclude the claim also for $A$.
\end{proof}

Now in order to show $\ker (A - \iu \omega) = \set{0}$ we make the following observation for supposed eigenvalues.

\begin{lemma}\label{le:eigenvectors-zero-trace}
  If $\iu \omega \in \iu \R$ is an eigenvalue of $A$, then the corresponding eigenvector
  \(
  x =
  \begin{bsmallmatrix}
    \vec{D} \\ \vec{B}
  \end{bsmallmatrix}
  \)
  satisfies $\tantr[\Gamma_{1}] \epsilon^{-1} \vec{D} = 0$ and $\tanxtr[\Gamma_{1}] \mu^{-1} \vec{B} = 0$.
\end{lemma}

\begin{proof}
  By \Cref{le:int-by-parts-for-A} and the boundary condition of $A$ ($\tantr \mu^{-1} \vec{D} = -k \tanxtr \epsilon^{-1} \vec{D}$ on $\Gamma_{1}$), we have
  \begin{align*}
    \Re \scprod{\smash[b]{\underbrace{(A-\iu \omega)x}_{=\mathrlap{0}}}}{x}_{\XH}
    &= \Re \scprod{Ax}{x}_{\XH} + \Re \iu \omega \scprod{x}{x}_{\XH} \\
    &= \Re \scprod{\tantr \epsilon^{-1} \vec{D}}{\tanxtr \mu^{-1} \vec{B}}_{\Lp{2}(\Gamma_{1})} \\
    &= \Re \scprod{\tantr \epsilon^{-1} \vec{D}}{-k \tantr \epsilon^{-1} \vec{D}}_{\Lp{2}(\Gamma_{1})} \\
    &= - \Re \norm{k^{\nicefrac{1}{2}} \tantr \epsilon^{-1} \vec{D}}_{\Lp{2}(\Gamma_{1})}.
  \end{align*}
  Hence, $\tantr[\Gamma_{1}] \epsilon^{-1} \vec{D} = 0$ and by the boundary condition we also have $\tanxtr[\Gamma_{1}] \mu^{-1} \vec{B} = 0$.
\end{proof}

Since the unique continuation principle from \Cref{sec:unique-continuation} only works for $\omega \neq 0$ we need to check $0 \in \uprho(A)$ separately.

\begin{proposition}\label{th:0-in-resolvent-set}
  $0 \in \uprho(A)$.
\end{proposition}

\begin{proof}
  Note that by \Cref{th:operator-with-compact-resolvent-only-eigenvalues} the spectrum $\upsigma(A)$ contains only eigenvalues.
  Hence, if we assume that $0 \not\in \uprho(A)$, then $0$ is an eigenvalue. Then by \Cref{le:eigenvectors-zero-trace} the corresponding eigenvector
  \(
  x
  =
  \begin{bsmallmatrix}
    \vec{D} \\ \vec{B}
  \end{bsmallmatrix}
  \)
  satisfies $\tantr[\Gamma_{1}] \epsilon^{-1} \vec{D} = 0 = \tanxtr[\Gamma_{1}] \mu^{-1} \vec{B}$. Combined with the remaining boundary condition for $\vec{D}$ we obtain $\tantr \epsilon^{-1} \vec{D} = 0$ on $\partial\Omega$. Thus, for $\phi \in \Hspace(\rot,\Omega)$, using $Ax = 0$, we obtain
  \begin{align*}
    0 = \scprod*{\smash[b]{\underbrace{\rot \epsilon^{-1} \vec{D}}_{=\mathrlap{0}}}}{\phi}_{\Lp{2}(\Omega)} = \scprod*{\epsilon^{-1} \vec{D}}{\rot \phi}_{\Lp{2}(\Omega)}.
    \vphantom{\underbrace{\rot}_{=0}} 
  \end{align*}
  Hence, $\vec{D} \perp_{\epsilon^{-1}} \rot \Hspace(\rot,\Omega)$ and since $\vec{D} \in \rot \Hspace(\rot,\Omega)$ (by $x \in \XH$), we conclude $\vec{D} = 0$.

  Similarly, for $\psi \in \cH_{\Gamma_{0}}(\rot,\Omega)$ we have
  \begin{align*}
    0 = \scprod*{\rot \mu^{-1} \vec{B}}{\psi}_{\Lp{2}(\Omega)}
    = \scprod{\mu^{-1} \vec{B}}{\rot \psi}_{\Lp{2}(\Omega)},
  \end{align*}
  which implies that $\vec{B} \perp_{\mu^{-1}} \ran \cH_{\Gamma_{0}}(\rot,\Omega)$. Since $\vec{B} \in \rot \cH_{\Gamma_{0}}(\rot,\Omega)$ we conclude $\vec{B} = 0$, which leads to $x = 0$. Therefore, $0$ is not an eigenvalue and $0 \in \uprho(A)$.
\end{proof}

Now we have collected all the tools to prove that the dynamic part of Maxwell's equation is semi-uniformly stable.

\begin{theorem}\label{th:semigroup-semi-uniformly-stable}
  The operator $A$ (from \Cref{def:differential-operator-A}) generates a semi-uniformly stable semigroup $(T(t))_{t\geq 0}$.
\end{theorem}

\begin{proof}
  By \Cref{th:equivalences-of-semi-uniform-stability} we have already reduced the question to \Cref{pro:reduced-question}, i.e., we have to show $\iu \R \subseteq \uprho(A)$.
  By \Cref{th:A-has-compact-resolvent}, $A$ has a compact resolvent, which implies that the spectrum $\upsigma(A)$ contains only eigenvalues (\Cref{th:operator-with-compact-resolvent-only-eigenvalues}).

  If $\iu w \in \iu \R \setminus \set{0}$ is an eigenvalue of $A$, then by \Cref{le:eigenvectors-zero-trace} the corresponding eigenvector has vanishing tangential trace on $\Gamma_{1}$.
  From the principle of unique continuation (\Cref{th:principle-of-unique-continuation}) follows that this eigenvector can only be $0$.
  Hence, there are no non-zero imaginary eigenvalues and consequently $\iu \R \setminus \set{0} \subseteq \uprho(A)$.

  Finally, \Cref{th:0-in-resolvent-set} yields $0 \in \uprho(A)$. Thus $\iu \R \subseteq \uprho(A)$ and the semigroup that is generated by $A$ is semi-uniformly stable.
  %
\end{proof}

Since every solution of the system~\eqref{eq:maxwells-eq} can be decomposed into an equilibrium and a dynamic part (that solves the abstract Cauchy problem corresponding to $A$), we conclude the main theorem.

\maintheorem*

Recall that
\begin{align*}
  \begin{bmatrix}
    \cohom_{\epsilon,\partial\Omega,\emptyset}(\Omega) \\ \cohom_{\mu,\Gamma_{1},\Gamma_{0}}(\Omega)
  \end{bmatrix}
  =
  \hamiltonian^{-1}
  \begin{bmatrix}
    \cohom_{\partial\Omega,\emptyset,\epsilon}(\Omega) \\ \cohom_{\Gamma_{1},\Gamma_{0},\mu}(\Omega)
  \end{bmatrix}
  \quad\text{and}\quad
  x \perp_{\hamiltonian^{-1}} y
  \;\Leftrightarrow\;
  \hamiltonian^{-1} x \perp_{\hamiltonian} \hamiltonian^{-1} y.
\end{align*}

\begin{proof}
  Let
  \(
  \begin{bsmallmatrix}
    \vec{E}_{\mathrm{e}} \\ \vec{H}_{\mathrm{e}}
  \end{bsmallmatrix}
  \)
  be a solution of~\eqref{eq:maxwell-static} with $h = \tanxtr \vec{H}$, i.e., an equilibrium state. Then we define the initial value for the ``dynamic'' part by
  \begin{align*}
    \begin{bmatrix}
      \vec{D}_{0} \\ \vec{B}_{0}
    \end{bmatrix}
    \coloneqq
    \underbrace{\begin{bmatrix}
      \epsilon & 0 \\
      0 & \mu
    \end{bmatrix}}_{=\mathrlap{\hamiltonian^{-1}}}
    \begin{bmatrix}
      \vec{E}_{0} - \vec{E}_{\mathrm{e}} \\
      \vec{H}_{0} - \vec{H}_{\mathrm{e}}
    \end{bmatrix}.
  \end{align*}
  Note that
  \(
  \begin{bsmallmatrix}
    \vec{E}_{\mathrm{e}} \\ \vec{H}_{\mathrm{e}}
  \end{bsmallmatrix}
  \)
  is in general---depending on the cohomology groups---not unique. Moreover,
  \(
  \begin{bsmallmatrix}
    \vec{D}_{0} \\ \vec{B}_{0}
  \end{bsmallmatrix}
  \)
  satisfies $\div \vec{D}_{0} = \div \vec{B}_{0} = 0$ and $\normaltr[\Gamma_{0}] \vec{B}_{0} = 0$, but is not necessarily in
  \(
  \begin{bsmallmatrix}
    \cohom_{\epsilon,\partial\Omega,\emptyset}(\Omega) \\ \cohom_{\mu,\Gamma_{1},\Gamma_{0}}(\Omega)
  \end{bsmallmatrix}^{\perp_{\hamiltonian}}
  \).
  Hence, we shift
  \(
  \begin{bsmallmatrix}
    \vec{E}_{\mathrm{e}} \\ \vec{H}_{\mathrm{e}}
  \end{bsmallmatrix}
  \)
  by the
  \(
  \begin{bsmallmatrix}
    \cohom_{\epsilon,\partial\Omega,\emptyset}(\Omega) \\ \cohom_{\mu,\Gamma_{1},\Gamma_{0}}(\Omega)
  \end{bsmallmatrix}
  \)
  part of
  \(
  \begin{bsmallmatrix}
    \vec{E}_{0} - \vec{E}_{\mathrm{e}} \\ \vec{H}_{0} - \vec{H}_{\mathrm{e}}
  \end{bsmallmatrix}
  \)
  w.r.t.\ $\scprod{\cdot}{\cdot}_{\hamiltonian^{-1}}$ such that
  \(
  \begin{bsmallmatrix}
    \vec{E}_{0} - \vec{E}_{\mathrm{e}} \\ \vec{H}_{0} - \vec{H}_{\mathrm{e}}
  \end{bsmallmatrix}
  \in
  \begin{bsmallmatrix}
    \cohom_{\partial\Omega,\emptyset,\epsilon}(\Omega) \\ \cohom_{\Gamma_{1},\Gamma_{0},\mu}(\Omega)
  \end{bsmallmatrix}^{\perp_{\hamiltonian^{-1}}}
  \)
  and consequently (the updated)
  \(
  \begin{bsmallmatrix}
    \vec{D}_{0} \\ \vec{B}_{0}
  \end{bsmallmatrix}
  \)
  is in
  \(
  \begin{bsmallmatrix}
    \cohom_{\epsilon,\partial\Omega,\emptyset}(\Omega) \\ \cohom_{\mu,\Gamma_{1},\Gamma_{0}}(\Omega)
  \end{bsmallmatrix}^{\perp_{\hamiltonian}}
  \).
  Hence,
  \(
  \begin{bsmallmatrix}
    \vec{D}_{0} \\ \vec{B}_{0}
  \end{bsmallmatrix}
  \in \XH
  \)
  and even in $\dom A$.

  We denote the semi-uniformly stable semigroup that is generated by $A$ by $(T(t))_{t \geq 0}$ and the decay function by $f$ (\Cref{th:semigroup-semi-uniformly-stable}). Hence, the solution
  \(
    \begin{bsmallmatrix}
      \vec{D} \\ \vec{B}
    \end{bsmallmatrix}
    (t,\cdot)
    \coloneqq
    T(t)
    \begin{bsmallmatrix}
      \vec{D}_{0} \\ \vec{B}_{0}
    \end{bsmallmatrix}
  \)
  of the dynamic part satisfies
  \begin{align*}
    \norm*{
    \begin{bmatrix}
      \vec{D} \\ \vec{B}
    \end{bmatrix}
    (t,\cdot)
    }_{\Lp{2}(\Omega)}
    \leq
    f(t)
    \norm*{
    \begin{bmatrix}
      \vec{D}_{0} \\ \vec{B}_{0}
    \end{bmatrix}
    }_{\dom A}
  \end{align*}
  The solution of~\eqref{eq:maxwells-eq} is then given by
  \begin{align*}
    \begin{bmatrix}
      \vec{E} \\ \vec{H}
    \end{bmatrix}
    (t,\cdot)
    =
    \begin{bmatrix}
      \epsilon^{-1} \vec{D} \\ \mu^{-1} \vec{B}
    \end{bmatrix}
    (t,\cdot)
    +
    \begin{bmatrix}
      \vec{E}_{\mathrm{e}} \\ \vec{H}_{\mathrm{e}}
    \end{bmatrix}.
  \end{align*}
  Note that $\hamiltonian = \begin{bsmallmatrix} \epsilon^{-1} & 0 \\ 0 & \mu^{-1} \end{bsmallmatrix}$ is a bounded and boundedly invertible operator, which implies that $c^{-1}\norm{x}_{\Lp{2}(\Omega)} \leq \norm{\hamiltonian x}_{\Lp{2}(\Omega)} \leq c \norm{x}_{\Lp{2}(\Omega)}$ for a suitable $c > 0$. Therefore,
  \begin{align*}
    \MoveEqLeft
    \norm*{
    \begin{bmatrix}
      \vec{E} \\ \vec{H}
    \end{bmatrix}
    (t,\cdot)
    -
    \begin{bmatrix}
      \vec{E}_{\mathrm{e}} \\ \vec{H}_{\mathrm{e}}
    \end{bmatrix}
    }_{\Lp{2}} \\
    &=
    \norm*{
    \begin{bmatrix}
      \epsilon^{-1} \vec{D} \\ \mu^{-1} \vec{B}
    \end{bmatrix}
    }_{\Lp{2}}
    \leq c
    \norm*{
    \begin{bmatrix}
      \vec{D} \\ \vec{B}
    \end{bmatrix}
      }_{\Lp{2}}
    \leq c f(t)
      \left(
      \norm*{
      \begin{bmatrix}
        \vec{D}_{0} \\ \vec{B}_{0}
      \end{bmatrix}
      }_{\Lp{2}}
      +
      \norm*{
      \begin{bmatrix}
        \rot \vec{E}_{0} \\ \rot \vec{H}_{0}
      \end{bmatrix}
      }_{\Lp{2}}
      \right)
    \\
    &\leq
      c(c+1)f(t)
      \left(
      \norm*{
      \begin{bmatrix}
        \vec{E}_{0} \\ \vec{H}_{0}
      \end{bmatrix}
      -
      \begin{bmatrix}
        \vec{E}_{\mathrm{e}} \\ \vec{H}_{\mathrm{e}}
      \end{bmatrix}
      }_{\Lp{2}}
      +
      \norm*{
      \begin{bmatrix}
        \rot \vec{E}_{0} \\ \rot \vec{H}_{0}
      \end{bmatrix}
      }_{\Lp{2}}
      \right).
  \end{align*}
  Redefining $f$ as $c(c+1)f$ finishes the proof.
\end{proof}


\section{Conclusion}

We have shown that Maxwell's equation without damping terms can be semi-uniformly stabilized by a simple boundary feedback. This was done without any geometric assumptions on the boundary like the geometric control condition, convexity, star-shapedness, etc.\ and with only mild conditions on the matrix-valued functions $\epsilon$ and $\mu$, i.e., strict positivity and Lipschitz continuity.

The key ingredients were a compact resolvent and a unique continuation principle.
The other arguments do not really depend on our particular differential operator, but can also be done for an entire class of systems. In particular the port-Hamiltonian systems that are discussed in \cite{Sk21,Sk-Phd}. In fact, if there were a generalization of the compact resolvent and the unique continuation principle for those port-Hamiltonian systems, we could conclude the same semi-uniform stability result.

\appendix

\section{Unique continuation}\label{sec:unique-continuation}

For $\epsilon$ and $\mu$ as in the beginning and $\omega \neq 0$ we regard the following stationary system,
\begin{align}\label{eq:eigenvector-equation}
  \begin{split}
    \rot \vec{H} &= \iu \omega \epsilon \vec{E}, \\
    \rot \vec{E} &= - \iu \omega \mu \vec{H}.
  \end{split}
\end{align}

Recall that $\epsilon$ and $\mu$ are strictly positive matrix-valued functions. In particular, there exists a $c > 0$ such that
\begin{align*}
  c^{-1} \leq \epsilon(\zeta) \leq c
  \quad\text{and}\quad
  c^{-1} \leq \mu(\zeta) \leq c
  \quad\text{for all}\quad \zeta \in \Omega
\end{align*}
and there exists a $C > 0$ such that
\begin{align*}
  \norm{\epsilon}_{\sobolev{1,\infty}(\Omega)}
  + \norm{\mu}_{\sobolev{1,\infty}(\Omega)}
  \leq C,
\end{align*}
where $\sobolev{1,\infty}(\Omega)$ denotes the Sobolev space of Lipschitz continuous functions.

The next theorem is from~\cite[Thm.~1.1]{NgWa12}.
\begin{theorem}\label{th:3-ball-estimate}
  Let $\vec{H}, \vec{E} \in \Lp{2}_{\textup{loc}}(\Omega)$ be a solution of~\eqref{eq:eigenvector-equation}.
  Then there exist $\rho, s >0$ such that for $r_{0} < r_{1} < r_{2}/2 < \rho$ with $\ball_{r_{2}}(x_{0}) \subseteq \Omega$, we have
  \begin{align*}
    \int_{\ball_{r_{1}}(x_{0})}
    \norm[\big]{
    \begin{bsmallmatrix}
      \vec{E} \\ \vec{H}
    \end{bsmallmatrix}}^{2}
    \dx[\uplambda]
    \leq
    C
    \left(
    \int_{\ball_{r_{0}}(x_{0})}
    \norm[\big]{
    \begin{bsmallmatrix}
      \vec{E} \\ \vec{H}
    \end{bsmallmatrix}}^{2}
    \dx[\uplambda]
    \right)^{\tau}
    \left(
    \int_{\ball_{r_{2}}(x_{0})}
    \norm[\big]{
    \begin{bsmallmatrix}
      \vec{E} \\ \vec{H}
    \end{bsmallmatrix}}^{2}
    \dx[\uplambda]
    \right)^{1-\tau},
  \end{align*}
  where $C, \tau > 0$ depend on $\epsilon,\mu,r_{1},r_{2},s$.
\end{theorem}

\begin{theorem}[Principle of unique continuation]%
  \label{th:principle-of-unique-continuation}
  Let $\Omega$ be connected and $\vec{E}, \vec{H} \in \Lp{2}_{\textup{loc}}(\Omega)$ be a solution of~\eqref{eq:eigenvector-equation}
  such that
  \(
  \begin{bsmallmatrix}
    \vec{E} \\ \vec{H}
  \end{bsmallmatrix}\big\vert_{U}
  = 0
  \)
  for a non-empty open set $U \subseteq \Omega$. Then
  \(
  \begin{bsmallmatrix}
    \vec{E} \\ \vec{H}
  \end{bsmallmatrix}
  = 0
  \).
\end{theorem}

\begin{proof}
  We define
  \begin{align*}
    M \coloneqq \dset*{\zeta \in \Omega}{\exists r>0 \ \text{s.t.}\
    \begin{bsmallmatrix}
      \vec{E} \\ \vec{H}
    \end{bsmallmatrix}
    \big\vert_{\ball_{r}(\zeta)}
    = 0
    }
  \end{align*}
  and show that $M$ is open and closed in $\Omega$.

  \medskip
  Clearly, $M$ is open in $\Omega$: By definition, for $\zeta_{0} \in M$ there exists an $r_{0}>0$ such that
  \(
  \begin{bsmallmatrix}
    \vec{E} \\
    \vec{H}
  \end{bsmallmatrix}\big\vert_{\ball_{r_{0}}(\zeta_{0})}
  = 0
  \).
  Hence, for every $\zeta \in \ball_{r_{0}}(\zeta_{0})$ we have
  \(
  \begin{bsmallmatrix}
    \vec{E} \\
    \vec{H}
  \end{bsmallmatrix}\big\vert_{\ball_{r_{0} - \abs{\zeta - \zeta_{0}}}(\zeta)}
  = 0
  \),
  which implies $\zeta \in M$ and therefore $\ball_{r_{0}}(\zeta_{0}) \subseteq M$.

  \begin{figure}[h]
    \centering
    \begin{tikzpicture}[scale=0.7]
      \coordinate[label=below:$\zeta$] (xi) at (0,0);
      \coordinate[label=below:$\zeta_{n_{0}}$] (xi0) at ($(xi) + 1.2*(1,0)$);
      \filldraw (xi) circle (0.1em);
      \filldraw (xi0) circle (0.1em);

      \draw (xi) circle (6);
      \draw (xi0) circle (2.4);
      \draw (xi0) circle (1);
      \draw (xi0) circle (4.2);

      \draw[] (xi) -- node[above,xshift=-5] {$r$} ($(xi) + 6*(-1,0)$);
      \draw[] (xi0) -- node[right=-2] {$r_{0}$} ($(xi0) + 1*(0,1)$);
      \draw[] (xi0) -- node[right=-2,yshift=3] {$r_{1}$} ($(xi0) + 2.4*({cos(120)},{sin(120)})$);
      \draw[] (xi0) -- node[above,xshift=-7,yshift=5] {$r_{2}$} ($(xi0) + 4.2*({cos(150)},{sin(150)})$);
      \draw[thick] plot[smooth,tension=0.5] coordinates {(-5.5,-6) (-5.5,-3) (-6.3,0) (-5,6)};
      \node at (-5.9,-3) {$\partial \Omega$};
    \end{tikzpicture}
    \caption{Illustration of the proof of \Cref{th:principle-of-unique-continuation}}
    \label{fig:proof-unique-continuation}
  \end{figure}
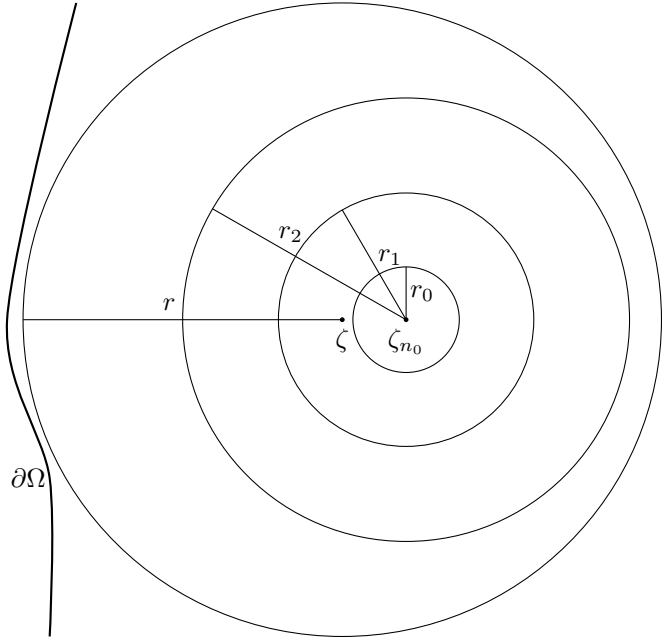

  \medskip
  On the other hand: Let $(\zeta_{n})_{n\in\N}$ be a sequence in $M$ that converges to $\zeta \in \Omega$.
  For every $\zeta_{n} \in M$ there exists an $r_{n} > 0$ such that
  \(
  \begin{bsmallmatrix}
    \vec{E} \\ \vec{H}
  \end{bsmallmatrix}
  \big\vert_{\ball_{r_{n}}(\zeta_{n})}
  = 0
  \).
  If there exists an $n\in\N$ such that $\zeta \in \ball_{r_{n}}(\zeta_{n})$, then clearly $\zeta \in M$. Hence we may assume that $r_{n} \leq \abs{\zeta_{n} - \zeta}$ for the remaining proof.

  Let $\rho$ be the number of \Cref{th:3-ball-estimate}. We choose $r > 0$ such that $\ball_{r}(\zeta) \subseteq \Omega$ and we choose $n_{0} \in \N$ such that $\abs{\zeta_{n_{0}} - \zeta} < \frac{1}{8}\min(r,\rho)$.
  We define (as illustrated in \Cref{fig:proof-unique-continuation})
  \begin{align*}
    r_{0} = r_{n_{0}},\quad
    r_{1} = 2 \abs{\zeta_{n_{0}} - \zeta}, \quad
    r_{2} = 6 \abs{\zeta_{n_{0}} - \zeta},
  \end{align*}
  which implies
  \begin{align*}
    r_{0} < r_{1} < r_{2}/2 < \rho.
  \end{align*}
  By \Cref{th:3-ball-estimate} we conclude
  \(
  \begin{bsmallmatrix}
    \vec{E} \\ \vec{H}
  \end{bsmallmatrix}
  \big\vert_{\ball_{r_{1}}(\zeta_{n_{0}})}
  = 0
  \)
  and consequently
  \(
  \begin{bsmallmatrix}
    \vec{E} \\ \vec{H}
  \end{bsmallmatrix}
  \big\vert_{\ball_{r_{1}/2}(\zeta)}
  = 0
  \),
  which implies $\zeta \in M$ and $M$ is closed.
  By assumption $M$ is non-empty, because $\emptyset \neq U \subseteq M$, hence we conclude that $M = \Omega$, as $\Omega$ is connected.
\end{proof}

\begin{proposition}
  Let
  \(
  \begin{bsmallmatrix}
    \vec{E} \\ \vec{H}
  \end{bsmallmatrix}
  \in \Lp{2}_{\textup{loc}}(\Omega)
  \)
  be a solution of~\eqref{eq:eigenvector-equation}.
  If there exists a $\Gamma \subseteq \partial \Omega$ relatively open such that $\tantr[\Gamma] \vec{E} = 0 = \tantr[\Gamma] \vec{H}$, then
  \(
  \begin{bsmallmatrix}
    \vec{E} \\ \vec{H}
  \end{bsmallmatrix}
  = 0
  \)
\end{proposition}

\begin{proof}
  We choose a point $\xi$ in the interior of $\Gamma$ and a radius $r > 0$ such that also $\ball_{r}(\xi) \cap \partial \Omega$ is in the interior of $\Gamma$. Then we extend
  \(
  \begin{bsmallmatrix}
    \vec{E} \\ \vec{H}
  \end{bsmallmatrix}
  \)
  on $\hat{\Omega} = \Omega \cup \ball_{r}(\xi)$ by
  \begin{align*}
    \hat{\vec{E}} =
    \begin{cases}
      \vec{E}, & \text{on}\ \Omega, \\
      0, & \text{on}\ \ball_{r}(\xi)\setminus\Omega,
    \end{cases}
           \quad\text{and}\quad
           \hat{\vec{H}} =
    \begin{cases}
      \vec{H}, & \text{on}\ \Omega, \\
      0, & \text{on}\ \ball_{r}(\xi)\setminus\Omega,
    \end{cases}
  \end{align*}
  Note that $\hat{\vec{E}}$ and $\hat{\vec{H}}$ are still in $\Hspace(\rot,\Omega)$, because for $\phi \in \Cc(\hat{\Omega})$
  \begin{align*}
    \scprod{\hat{\vec{E}}}{\rot \phi}_{\Lp{2}(\hat{\Omega})}
    &= \scprod{\vec{E}}{\rot \phi}_{\Lp{2}(\Omega)}
    = \scprod{\rot \vec{E}}{\phi}_{\Lp{2}(\Omega)} + \scprod{\underbrace{\tanxtr \vec{E}}_{=\mathrlap{0}}}{\tantr \phi}_{\Lp{2}(\Gamma)}
    \\
    &= \scprod{\widehat{\rot \vec{E}}}{\phi}_{\Lp{2}(\hat{\Omega})},
  \end{align*}
  where
  \begin{align*}
    \widehat{\rot \vec{E}} =
    \begin{cases}
      \rot \vec{E}, & \text{on}\ \Omega, \\
      0, & \text{on}\ \ball_{r}(\xi)\setminus\Omega.
    \end{cases}
  \end{align*}
  We can prove the same for $\vec{H}$ and obtain that
  \(
  \begin{bsmallmatrix}
    \hat{\vec{E}} \\ \hat{\vec{H}}
  \end{bsmallmatrix}
  \)
  is a solution of~\eqref{eq:eigenvector-equation} (on $\hat{\Omega}$). Since
  \(
  \begin{bsmallmatrix}
    \hat{\vec{E}} \\ \hat{\vec{H}}
  \end{bsmallmatrix}
  \big\vert_{\ball_{r}(\xi) \setminus \Omega}
  = 0
  \),
  we can apply \Cref{th:principle-of-unique-continuation} which leads to
  \(
  \begin{bsmallmatrix}
    \hat{\vec{E}} \\ \hat{\vec{H}}
  \end{bsmallmatrix}
  = 0
  \).
\end{proof}




\section{Background on the differential and trace operators}\label{sec:background-diffoperator-and-traces}

In this section we want to give a little background on the operator $A_{0}$ defined in~\eqref{eq:def-A-0} and ultimately justify that it generates a contraction semigroup. We still have the standing assumption that $(\Omega, \Gamma_{1})$ is a strongly Lipschitz pair.  In order to do that we have to explain the boundary spaces and traces that correspond to $\rot$. We recall the operator $A_{0}$ from \eqref{eq:def-A-0}
\begin{align*}
  A_{0}
  &=
  \begin{bmatrix}
    0 & \rot \\
    -\rot & 0
  \end{bmatrix}
            \underbrace{
            \begin{bmatrix}
              \epsilon^{-1} & 0 \\
              0 & \mu^{-1}
            \end{bmatrix}
                  }_{=\mathrlap{\hamiltonian}}
  \\
  \dom A_{0}
  &=
    \left\{
    \begin{bmatrix}
      \vec{D} \\ \vec{B}
    \end{bmatrix}
    \in
    \hamiltonian^{-1}
    \begin{bmatrix}
      \hH_{\partial\Omega}(\rot,\Omega) \cap \cH_{\Gamma_{0}}(\rot,\Omega)\\ \hH_{\Gamma_{1}}(\rot,\Omega)
    \end{bmatrix}
    \right.
    \hphantom{some extra space xxxxx}
  \\
  &\pushright{
    \left|\,
    \vphantom{\begin{bmatrix}\vec{D} \\ \vec{B}\end{bmatrix}}
    \tantr[\Gamma_{1}] \epsilon^{-1} \vec{D} + k \tanxtr[\Gamma_{1}] \mu^{-1} \vec{B} = 0
    \right\}
    }.
\end{align*}
In order to better understand the domain of $A_{0}$ we have to take a closer look at the traces $\tantr[\Gamma_{0}]$, $\tantr[\Gamma_{1}]$ and $\tanxtr[\Gamma_{1}]$.
First of all note that the differential operator $\rot$ can be written as
\begin{align}
  \begin{split}\label{eq:decompose-rot}
    \rot f
    &=
      \begin{bmatrix}
        0 & -\partial_{3} & \partial_{2} \\
        \partial_{3} & 0 & -\partial_{1} \\
        -\partial_{2} & \partial_{1} & 0
      \end{bmatrix}
      \begin{bmatrix}
        f_{1} \\ f_{2} \\ f_{3}
      \end{bmatrix}
    \\
    &=
      \vphantom{\partial_{1}
      \underbrace{
      \begin{bmatrix}
        0 & 0 & 0 \\
        0 & 0 & -1 \\
        0 & 1 & 0
      \end{bmatrix}
      }_{\eqqcolon\mathrlap{L_{1}}}
      }
      \left(
      \smash[b]{
      \partial_{1}
      \underbrace{
      \begin{bmatrix}
        0 & 0 & 0 \\
        0 & 0 & -1 \\
        0 & 1 & 0
      \end{bmatrix}
      }_{\eqqcolon\mathrlap{L_{1}}}
      \mathclose{}+\mathopen{}
      \partial_{2}
      \underbrace{
      \begin{bmatrix}
        0 & 0 & 1 \\
        0 & 0 & 0 \\
        -1 & 0 & 0
      \end{bmatrix}
      }_{\eqqcolon \mathrlap{L_{2}}}
      \mathclose{}+\mathopen{}
      \partial_{3}
      \underbrace{
      \begin{bmatrix}
        0 & -1 & 0 \\
        1 & 0 & 0 \\
        0 & 0 & 0
      \end{bmatrix}
      }_{\eqqcolon\mathrlap{L_{3}}}
      }
      \right)
      \begin{bmatrix}
        f_{1} \\ f_{2} \\ f_{3}
      \end{bmatrix}.
  \end{split}
\end{align}
Therefore, it has the form $\rot = \diffop \coloneqq \sum_{i=1}^{3} \partial_{i} L_{i}$.
Hence, the operator matches the form of the differential operators in~\cite{Sk21}. That work presents a general approach to boundary traces and boundary spaces that correspond to such differential operators. These boundary spaces and traces can be derived by an integration by parts formula.
Nevertheless we want to present here a sketch of this construction that is adjusted just for the $\rot$ operator.

For the $\rot$ operator we have for $f,g \in \Cc(\R^{3})$ the following integration by parts formula
\begin{align}
  \label{eq:int-by-parts-rot}
  \scprod{\rot f}{g}_{\Lp{2}(\Omega)} + \scprod{f}{-\rot g}_{\Lp{2}(\Omega)} = \scprod{\nu \times f}{(\nu \times g) \times \nu}_{\Lp{2}(\partial\Omega)}.
\end{align}
We can restrict ourselves to the boundary space $\Lp{2}_{\tau}(\partial\Omega) \coloneqq \dset{\phi \in \Lp{2}(\partial\Omega)}{\nu \cdot \phi = 0}$, as both arguments of the $\Lp{2}(\partial\Omega)$ inner product in \eqref{eq:int-by-parts-rot} belong to that space anyway.

In the following we want to show that similar to the integration by parts formula for $\div$-$\grad$ we can extend the integration parts formula by continuity on the maximal domain of the differential operator. The price to pay is that the $\Lp{2}$ inner product on the boundary has to be replaced by a dual pairing. For $\div$-$\grad$ this would be the dual pairing of $(\Hspace^{\nicefrac{1}{2}}(\partial\Omega),\Hspace^{-\nicefrac{1}{2}}(\partial\Omega))$, which forms a Gelfand triple with the pivot space $\Lp{2}(\partial\Omega)$.
Unfortunately unlike in the $\div$-$\grad$ case the boundary spaces that correspond to $\rot$ do not establish a Gelfand triple, at least not in the usual sense where we have continuous embeddings.
However, we get something that is almost a Gelfand triple, what we call \emph{quasi Gelfand triple}, a notion that was introduced in~\cite{Sk21} or more detailed in~\cite{Sk-Phd,Sk23}.

In particular we will give a sketch of the construction of the boundary spaces and traces that correspond to $\rot$.

\subsection{Boundary spaces}
Note that the integration by parts formula \eqref{eq:int-by-parts-rot} can easily be extended to $\Hspace^{1}(\Omega)$. Hence, we want to take the step to extend \eqref{eq:int-by-parts-rot} from $\Hspace^{1}(\Omega)$ to $\Hspace(\rot,\Omega)$.
In order to do this we introduce the spaces
\begin{align*}
  M(\Gamma_{1}) &= \dset*{(\nu \times \boundtr g \big\vert_{\Gamma_{1}})\times \nu}{g \in \Hspace^{1}(\Omega)} \subseteq \Lp{2}_{\tau}(\Gamma_{1}), \\
  \mathring{M}(\Gamma_{1}) &= \dset*{(\nu \times \boundtr g \big\vert_{\Gamma_{1}})\times \nu}{g \in \cH^{1}_{\Gamma_{0}}(\Omega)} \subseteq \Lp{2}_{\tau}(\Gamma_{1}).
\end{align*}
Note that every element of $\mathring{M}(\Gamma_{1})$ can be extended to $M(\partial\Omega)$ by setting it to zero on $\Gamma_{0}$. We define the spaces $\Vtau(\Gamma_{1})$ and $\cVtau(\Gamma_{1})$ as the completions of $M(\Gamma_{1})$ and $\mathring{M}(\Gamma_{1})$, respectively, with respect to the range norms
\begin{align*}
  \norm{\phi}_{\Vtau(\Gamma_{1})} &\coloneqq \inf_{\substack{g \in \Hspace^{1}(\Omega) \\ (\nu \times \boundtr g\vert_{\Gamma_{1}}) \times \nu = \phi}} \norm{g}_{\Hspace(\rot,\Omega)} \quad\text{for}\quad \phi \in M(\Gamma_{1}), \\
  \norm{\phi}_{\cVtau(\Gamma_{1})} &\coloneqq \inf_{\substack{g \in \cH^{1}_{\Gamma_{0}}(\Omega) \\ (\nu \times \boundtr g\vert_{\Gamma_{1}} ) \times \nu = \phi}} \norm{g}_{\Hspace(\rot,\Omega)} \quad\text{for}\quad \phi \in \mathring{M}(\Gamma_{1}),
\end{align*}
respectively. These two norms are really norms by \cite[Lem.~6.3]{Sk21}. By construction we have that
\begin{align*}
  \tantr[\Gamma_{1}]&\colon \mapping{\Hspace^{1}(\Omega) \subseteq \Hspace(\rot,\Omega)}{\Vtau(\Gamma_{1})}{g}{(\nu \times \boundtr g \big\vert_{\Gamma_{1}})\times \nu,} \\
  \mathllap{\quad\text{and}\quad}
  \ctantr[\Gamma_{1}]&\colon \mapping{\cH^{1}_{\Gamma_{0}}(\Omega) \subseteq \cH_{\Gamma_{0}}(\rot,\Omega)}{\cVtau(\Gamma_{1})}{g}{(\nu \times \boundtr g \big\vert_{\Gamma_{1}})\times \nu,}
\end{align*}
are continuous w.r.t.\ $\norm{\cdot}_{\Hspace(\rot,\Omega)}$ and $\norm{\cdot}_{\Vtau(\Gamma_{1})}$ ($\norm{\cdot}_{\cVtau(\Gamma_{1})}$).
Note that the restriction bar $\big\vert_{\Gamma_{1}}$ at $\tantr[\Gamma_{1}]$ is an abuse of notation. It indicates that we are only interested in what the trace does on $\Gamma_{1}$. Moreover, note that one characterization of $\cH_{\Gamma_{0}}(\rot,\Omega)$ is the closure of $\cH^{1}_{\Gamma_{0}}(\Omega)$ w.r.t.\ $\norm{\cdot}_{\Hspace(\rot,\Omega)}$, see e.g., \cite{BaPaScho16}.
Now we can extend $\tantr[\Gamma_{1}]$ and $\ctantr[\Gamma_{1}]$ by density and continuity to $\Hspace(\rot,\Omega)$ and $\cH_{\Gamma_{0}}(\rot,\Omega)$, respectively. We will use the same symbols for these extensions and call both of them \emph{tangential trace}. If $\Gamma_{1} = \partial\Omega$, then clearly $\Vtau(\partial\Omega) = \cVtau(\partial\Omega)$ and we will just write $\tantr$ instead of $\tantr[\partial\Omega]$ or $\ctantr[\partial\Omega]$.
Sometimes it is convenient to also leave out the circle on $\ctantr[\Gamma_{1}]$ even if we work with elements of $\cH_{\Gamma_{0}}(\rot,\Omega)$ that are mapped into $\cVtau(\Gamma_{0})$.
Moreover, if it is clear from the context, we will also just use $\tantr$ instead of $\tantr[\Gamma_{1}]$ and $\ctantr[\Gamma_{1}]$.

\begin{remark}
  For a short moment we want to distinguish between $\tantr[\Gamma_{1}]$ and its continuous extension on $\Hspace(\rot,\Omega)$, by denoting the extension by $\tantrex[\Gamma_{1}]$. Then it can be shown that $\tantrex[\Gamma_{1}] \colon \Hspace(\rot,\Omega) \to \Vtau(\Gamma_{1})$ is surjective and
  \begin{equation*}
    \ker \tantrex[\Gamma_{1}] = \cl[\Hspace(\rot,\Omega)]{\ker \tantr[\Gamma_{1}]}.
  \end{equation*}
  An analogous result holds for $\ctantr[\Gamma_{1}]$, see \cite[Lem.~6.4]{Sk21}.
\end{remark}

Basically we can repeat the previous construction for the twisted tangential trace $f \mapsto \nu \times \boundtr f$. Note that for smooth functions $g \mapsto (\nu \times g \big\vert_{\partial\Omega}) \times \nu$ is really the projection of $g$ on its tangential component. Therefore, the name \emph{tangential trace} is justified for $\tantr$. Furthermore, for smooth functions we have $\nu \times f \big\vert_{\partial\Omega} = \nu \times \tantr f$, which tells us that $\nu \times f \big\vert_{\partial\Omega}$ is a $90$ degree (or $\frac{\uppi}{2}$) rotated version of the tangential component of $f \big\vert_{\partial\Omega}$. The axis of rotation is the normal vector $\nu$.

Even though it is basically a repetition we will execute the construction also for the \emph{twisted tangential trace}. Hence, we define the twisted version of $M(\Gamma_{1})$ and $\mathring{M}(\Gamma_{1})$ as
\begin{align*}
  M^{\times}(\Gamma_{1}) &\coloneqq \dset*{\nu \times \boundtr f \big\vert_{\Gamma_{1}}}{f \in \Hspace^{1}(\Omega)} \subseteq \Lp{2}_{\tau}(\Gamma_{1}) \\
  \mathllap{\text{and}\quad}
  \mathring{M}^{\times}(\Gamma_{1}) &\coloneqq \dset*{\nu \times \boundtr f \big\vert_{\Gamma_{1}}}{f \in \cH_{\Gamma_{0}}^{1}(\Omega)} \subseteq \Lp{2}_{\tau}(\Gamma_{1}),
\end{align*}
respectively. Furthermore we define $\Vtaud(\Gamma_{1})$ and $\cVtaud(\Gamma_{1})$ as the completion of $M^{\times}(\Gamma_{1})$ and $\mathring{M}^{\times}(\Gamma_{1})$, respectively with respect to the range norms
\begin{align*}
  \norm{\psi}_{\Vtaud(\Gamma_{1})}
  &\coloneqq \inf_{\substack{f \in \Hspace^{1}(\Omega) \\ \nu \times \boundtr f \vert_{\Gamma_{1}} = \psi}}
  \norm{f}_{\Hspace(\rot,\Omega)} \quad\text{for}\quad \psi \in M^{\times}(\Gamma_{1}) \\
  \mathllap{\text{and}\quad}
  \norm{\psi}_{\cVtaud(\Gamma_{1})}
  &\coloneqq \inf_{\substack{f \in \cH^{1}_{\Gamma_{0}}(\Omega) \\ \nu \times \boundtr f \vert_{\Gamma_{1}} = \psi}}
  \norm{f}_{\Hspace(\rot,\Omega)} \quad\text{for}\quad \psi \in \mathring{M}^{\times}(\Gamma_{1}),
\end{align*}
respectively.
The cross symbol $\times$ in the superscript indicates that we deal with a twisted version of $M(\Gamma_{1})$ and $\Vtau(\Gamma_{1})$, respectively. In particular, the operation $\nu \times \cdot$ can be extended to a unitary mapping between $\Vtau(\Gamma_{1})$ and $\Vtaud(\Gamma_{1})$, and a unitary operator between $\cVtau(\Gamma_{1})$ and $\cVtaud(\Gamma_{1})$.
By construction the mapping
\begin{align*}
  \tanxtr[\Gamma_{1}]
  &\colon \mapping{\Hspace^{1}(\Omega) \subseteq \Hspace(\rot,\Omega)}{\Vtaud(\Gamma_{1})}{f}{\nu \times \boundtr f \big\vert_{\Gamma_{1}},} \\
  \mathllap{\text{and}\quad}
  \ctanxtr[\Gamma_{1}]
  &\colon \mapping{\cH_{\Gamma_{0}}^{1}(\Omega) \subseteq \cH_{\Gamma_{0}}(\rot,\Omega)}{\cVtaud(\Gamma_{1})}{f}{\nu \times \boundtr f \big\vert_{\Gamma_{1}},}
\end{align*}
is continuous w.r.t.\ $\norm{\cdot}_{\Hspace(\rot,\Omega)}$ and $\norm{\cdot}_{\Vtaud(\Gamma_{1})}$ ($\norm{\cdot}_{\cVtaud(\Gamma_{1})}$). Hence, we can extend these mappings by continuity and density to $\Hspace(\rot,\Omega)$ and $\cH_{\Gamma_{0}}(\rot,\Omega)$, respectively.
We still denote these extension by $\tanxtr[\Gamma_{1}]$ and $\ctanxtr[\Gamma_{1}]$.
If $\Gamma_{1} = \partial\Omega$ we will just write $\tanxtr$. Also if it is clear from the context that we only want to regard the trace on $\Gamma_{1}$, we will just write $\tanxtr$ instead of $\tanxtr[\Gamma_{1}]$.
We will sometimes also omit the circle at $\ctanxtr$, if is is clear that we regard elements in $\cH_{\Gamma_{0}}(\rot,\Omega)$.

Note that the $\tau$ in the subscript of $\Lp{2}_{\tau}(\Gamma_{1})$, $\cVtau(\Gamma_{1})$, $\Vtau(\Gamma_{1})$, $\cVtaud(\Gamma_{1})$ and $\Vtaud(\Gamma_{1})$ indicates that these spaces are tangential on $\Gamma_{1}$.

\subsection{Dual pairing} Next we want to show that there is a dual pairing between $\cVtau(\Gamma_{1})$ and $\Vtaud(\Gamma_{1})$, i.e., $(\cVtau(\Gamma_{1}),\Vtaud(\Gamma_{1}))$ is a dual pair. Clearly, the same holds for their twisted versions $(\Vtau(\Gamma_{1}),\cVtaud(\Gamma_{1}))$, which is then just a corollary.

\begin{lemma}\label{le:continuous-dual-pairing}
  Let $\phi \in \mathring{M}(\Gamma_{1})$ and $\psi \in M^{\times}(\Gamma_{1})$. Then
  \begin{align*}
    \abs{\scprod{\psi}{\phi}_{\Lp{2}_{\tau}(\Gamma_{1})}} \leq \norm{\psi}_{\Vtaud(\Gamma_{1})} \norm{\phi}_{\cVtau(\Gamma_{1})}.
  \end{align*}
\end{lemma}

\begin{proof}
  Note that $\mathring{M}(\Gamma_{1}) = \tantr[\Gamma_{1}] \cH_{\Gamma_{0}}^{1}(\Omega)$ and $M^{\times}(\Gamma_{1}) = \tanxtr[\Gamma_{1}] \Hspace^{1}(\Omega)$. Hence, there exists a $g \in \cH^{1}_{\Gamma_{0}}(\Omega)$ and an $f \in \Hspace^{1}(\Omega)$ such that $\tantr[\Gamma_{1}] g = \phi$ and $\tanxtr[\Gamma_{1}] f = \psi$. By $\boundtr g \big\vert_{\Gamma_{0}} = 0$, the integration by parts formula \eqref{eq:int-by-parts-rot} and Cauchy--Schwarz's inequality we have
  \begin{align*}
    \abs*{\scprod{\psi}{\phi}_{\Lp{2}_{\tau}(\Gamma_{1})}}
    &= \abs*{\scprod[\big]{\tanxtr[\Gamma_{1}] f}{\tantr[\Gamma_{1}] g}_{\Lp{2}_{\tau}(\Gamma_{1})}}
      = \abs*{\scprod[\big]{\tanxtr[\Gamma_{1}] f}{\tantr[\Gamma_{1}] g}_{\Lp{2}_{\tau}(\partial\Omega)}} \\
    &= \abs*{\scprod{\rot f}{g}_{\Lp{2}(\Omega)} - \scprod{f}{\rot g}_{\Lp{2}(\Omega)}} \\
    &\leq \norm{\rot f}_{\Lp{2}(\Omega)} \norm{g}_{\Lp{2}(\Omega)} + \norm{f}_{\Lp{2}(\Omega)} \norm{\rot g}_{\Lp{2}(\Omega)} \\
    &\leq \sqrt{\norm{\rot f}^{2}_{\Lp{2}(\Omega)} + \norm{f}^{2}_{\Lp{2}(\Omega)}} \sqrt{\norm{g}^{2}_{\Lp{2}(\Omega)} + \norm{\rot g}^{2}_{\Lp{2}(\Omega)}} \\
    &= \norm{f}_{\Hspace(\rot,\Omega)} \norm{g}_{\Hspace(\rot,\Omega)}.
  \end{align*}
  Since this is true for all $g \in \cH^{1}_{\Gamma_{0}}(\Omega)$ and $f \in \Hspace^{1}(\Omega)$ such that $\tantr[\Gamma_{1}] g = \phi$ and $\tanxtr[\Gamma_{1}] f = \psi$, we can apply an infimum on the right-hand side and obtain
  \begin{equation*}
    \abs{\scprod{\psi}{\phi}_{\Lp{2}_{\tau}(\Gamma_{1})}}
    \leq \inf \norm{f}_{\Hspace(\rot,\Omega)} \norm{g}_{\Hspace(\rot,\Omega)}
    = \norm{\psi}_{\Vtaud(\Gamma_{1})} \norm{\phi}_{\cVtau(\Gamma_{1})}. \qedhere
  \end{equation*}
\end{proof}

Now we can define the following dual pairing by a limit.

\begin{definition}
  Let $\phi \in \cVtau(\Gamma_{1})$ and $\psi \in \Vtaud(\Gamma_{1})$. Then there exist sequences $(\phi_{n})_{n\in\N}$ in $\mathring{M}(\Gamma_{1})$ and $(\psi_{k})_{k\in\N}$ in $M^{\times}(\Gamma_{1})$ that converge to $\phi$ and $\psi$, respectively. We define the dual pairing between $\cVtau(\Gamma_{1})$ and $\Vtaud(\Gamma_{1})$ by
  \begin{align}\label{eq:def-dual-pairing}
    \dualprod{\psi}{\phi}_{\Vtaud(\Gamma_{1}),\cVtau(\Gamma_{1})} \coloneqq \lim_{\substack{n\to \infty \\k \to \infty}} \scprod{\psi_{n}}{\phi_{k}}_{\Lp{2}_{\tau}(\Gamma_{1})}.
  \end{align}
  As short notation we will use
  \(
  \dualprod{\psi}{\phi}_{\Vtaud,\cVtau}
  \),
  if $\Gamma_{1}$ is clear.
\end{definition}

This dual pairing is really well-defined, since we can use \Cref{le:continuous-dual-pairing} to show that the net on the right-hand side is a Cauchy net. Strictly speaking the mapping $\dualprod{\cdot}{\cdot}_{\Vtaud(\Gamma_{1}),\cVtau(\Gamma_{1})}$ is a priori just a sesquilinear form. However, we will show that it is indeed a dual pairing, i.e., the mapping
\begin{align}\label{eq:dual-pairing-induced-isometry}
  \Psi\colon \mapping{\Vtaud(\Gamma_{1})}{\cVtau(\Gamma_{1})\dual}{\psi}{\dualprod{\psi}{\cdot}_{\Vtaud(\Gamma_{1}),\cVtau(\Gamma_{1})},}
\end{align}
is an isometric isomorphism. Note that we use the convention of always regarding the antidual space, which is more convenient when switching between dual pairings and inner products. Hence, $\cVtau(\Gamma_{1})\dual$ denotes the antidual space.

Moreover, for every such sesquilinear form we define the version with switched arguments by the complex conjugate of the original sesquilinear form, i.e.,
\begin{equation*}
  \dualprod{\phi}{\psi}_{\cVtau(\Gamma_{1}), \Vtaud(\Gamma_{1})} \coloneqq \conj{\dualprod{\psi}{\phi}_{\Vtaud(\Gamma_{1}),\cVtau(\Gamma_{1})}}.
\end{equation*}

However, before we show that we have indeed defined a dual pairing we show that we can extend the integration by parts formula with $\dualprod{\cdot}{\cdot}_{\Vtaud(\Gamma_{1}),\cVtau(\Gamma_{1})}$ for arbitrary $f \in \Hspace(\rot,\Omega)$ and $g \in \cH_{\Gamma_{0}}(\rot,\Omega)$.

\begin{lemma}\label{le:int-by-parts-extended-to-max-dom}
  For $f \in \Hspace(\rot,\Omega)$ and $g \in \cH_{\Gamma_{0}}(\rot,\Omega)$ we have
  \begin{equation*}
    \scprod{\rot f}{g}_{\Lp{2}(\Omega)} - \scprod{f}{\rot g}_{\Lp{2}(\Omega)} = \dualprod{\tanxtr f}{\tantr g}_{\Vtaud(\Gamma_{1}),\cVtau(\Gamma_{1})}.
  \end{equation*}
\end{lemma}

\begin{proof}
  Note that by \eqref{eq:int-by-parts-rot} we have for $f \in \Cc(\R^{3})$ and $g \in \Cc_{\Gamma_{0}}(\R^{3})$
  \begin{align*}
    \scprod{\rot f}{g}_{\Lp{2}(\Omega)} - \scprod{f}{\rot g}_{\Lp{2}(\Omega)} = \scprod{\tanxtr f}{\tantr g}_{\Lp{2}_{\tau}(\Gamma_{1})}
  \end{align*}
  Since $\Cc(\R^{3})$ is dense in $\Hspace(\rot,\Omega)$ and $\Cc_{\Gamma_{0}}(\R^{3})$ is dense in $\cH_{\Gamma_{0}}(\rot,\Omega)$, the assertion follows by continuity.
\end{proof}

Note that $\cVtau(\Gamma_{1})$ and $\Vtaud(\Gamma_{1})$ are by construction isometrically isomorphic to the quotient spaces $\quspace{\cH_{\Gamma_{0}}(\rot,\Omega)}{\ker \ctantr[\Gamma_{1}]}$ and $\quspace{\Hspace(\rot,\Omega)}{\ker \tanxtr[\Gamma_{1}]}$, respectively.
These spaces are in turn isometrically isomorphic to $(\ker \ctantr[\Gamma_{1}])^{\perp}$ and $(\ker \tanxtr[\Gamma_{1}])^{\perp}$, respectively, where the orthogonal complement is taken in $\cH_{\Gamma_{0}}(\rot,\Omega)$ and $\Hspace(\rot,\Omega)$, respectively. Hence, the mappings $\ctantr[\Gamma_{1}]$ and $\tanxtr[\Gamma_{1}]$ are isometric isomorphisms from these orthogonal complements into $\cVtau(\Gamma_{1})$ and $\Vtaud(\Gamma_{1})$, respectively. Moreover, we have
\begin{align*}
  \Cc(\Omega) \subseteq \ker \ctantr[\Gamma_{1}] \subseteq \ker \tanxtr[\Gamma_{1}].
\end{align*}
Hence, for the orthogonal complements (in $\Hspace(\rot,\Omega)$) we have the reverse inclusion. The next lemmas also show similarities to the notion of \emph{boundary data spaces} from \cite[Sec.~5.2]{PiTrWa16}.

\begin{lemma}\label{le:ortho-complement-of-Cc-Omega}
  The orthogonal complement of $\Cc(\Omega)$ in $\Hspace(\rot,\Omega)$ can be characterized by
  \begin{align*}
    \Cc(\Omega)^{\perp} = \dset{f \in \Hspace(\rot,\Omega)}{\rot\rot f = -f}.
  \end{align*}
\end{lemma}

\begin{proof}
  Let $f \in \Cc(\Omega)^{\perp}$. Then for every $g \in \Cc(\Omega)$ we have
  \begin{align*}
    0 = \scprod{f}{g}_{\Hspace(\rot,\Omega)} = \scprod{f}{g}_{\Lp{2}(\Omega)} + \scprod{\rot f}{\rot g}_{\Lp{2}(\Omega)}.
  \end{align*}
  Hence, $\rot f \in \Hspace(\rot,\Omega)$ and $\rot \rot f = - f$.

  On the other hand, if $\rot \rot f = -f$, then we have for every $g \in \Cc(\Omega)$
  \begin{align*}
    \scprod{f}{g}_{\Hspace(\rot,\Omega)}
    &= \scprod{f}{g}_{\Lp{2}(\Omega)} + \scprod{\rot f}{\rot g}_{\Lp{2}(\Omega)} \\
    &= \scprod{\rot (-\rot f)}{g}_{\Lp{2}(\Omega)} - \scprod{-\rot f}{\rot g}_{\Lp{2}(\Omega)} \\
    &= \scprod{\tanxtr (-\rot f)}{\tantr g}_{\Vtaud(\partial\Omega),\Vtau(\partial\Omega)}.
  \end{align*}
  Since $\tantr g = 0$ for every $g \in \Cc(\Omega)$ we conclude the assertion.
\end{proof}

\begin{corollary}\label{th:properties-of-preimage-of-Vtaud-elements}
  For every $\psi \in \Vtaud(\Gamma_{1})$ there exists an $f \in \Hspace(\rot,\Omega)$ such that
  \begin{align*}
    \psi = \tanxtr f, \quad \norm{\psi}_{\Vtaud(\Gamma_{1})} = \norm{f}_{\Hspace(\rot,\Omega)} = \norm{\rot f}_{\Hspace(\rot,\Omega)} \quad\text{and} \quad \rot \rot f = -f .
  \end{align*}
  Moreover, for every $g \in \cH_{\Gamma_{0}}(\rot,\Omega)$ we have
  \begin{align*}
    \scprod{\rot f}{g}_{\Hspace(\rot,\Omega)} = \dualprod{\psi}{\tantr g}_{\Vtaud(\Gamma_{1}),\cVtau(\Gamma_{1})}.
  \end{align*}
\end{corollary}

\begin{proof}
  Recall that $\Vtaud(\Gamma_{1})$ is isometrically isomorphic to $(\ker \tanxtr[\Gamma_{1}])^{\perp}$. Hence, for every $\psi \in \Vtaud(\Gamma_{1})$ there exists an $f \in (\ker \tanxtr[\Gamma_{1}])^{\perp} \subseteq \Cc(\Omega)^{\perp} \subseteq \Hspace(\rot,\Omega)$ such that $\tanxtr f = \psi$ and $\norm{\psi}_{\Vtaud(\Gamma_{1})} = \norm{f}_{\Hspace(\rot,\Omega)}$. By \Cref{le:ortho-complement-of-Cc-Omega} we have $\rot \rot f = -f$. Moreover, we have
  \begin{equation*}
    \norm{\rot f}_{\Hspace(\rot,\Omega)}^{2} = \norm{\rot f}_{\Lp{2}(\Omega)}^{2} + \norm{\underbrace{\rot \rot f}_{=\mathrlap{-f}}}_{\Lp{2}(\Omega)}^{2} = \norm{f}_{\Hspace(\rot,\Omega)}^{2} = \norm{\psi}_{\Vtaud(\Gamma_{1})}^{2}.
  \end{equation*}
  Finally, we have
  \begin{align*}
    \dualprod{\psi}{\tantr g}_{\Vtaud(\Gamma_{1}),\cVtau(\Gamma_{1})}
    &= \dualprod{\tanxtr f}{\tantr g}_{\Vtaud(\Gamma_{1}),\cVtau(\Gamma_{1})}
    = \scprod{\rot f}{g}_{\Lp{2}(\Omega)} - \scprod{f}{\rot g}_{\Lp{2}(\Omega)} \\
    &= \scprod{\rot f}{g}_{\Lp{2}(\Omega)} + \scprod{\rot \rot f}{\rot g}_{\Lp{2}(\Omega)} \\
    &= \scprod{\rot f}{g}_{\Hspace(\rot,\Omega)}. \tag*{\qedhere}
  \end{align*}
\end{proof}

\begin{theorem}
  The spaces $\cVtau(\Gamma_{1})$ and $\Vtaud(\Gamma_{1})$ form the dual pair $(\cVtau(\Gamma_{1}), \Vtaud(\Gamma_{1}))$ with the dual pairing $\dualprod{\cdot}{\cdot}_{\Vtaud(\Gamma_{1}),\cVtau(\Gamma_{1})}$ defined in~\eqref{eq:def-dual-pairing}.
\end{theorem}

\begin{proof}
  Note that the mapping $\Psi$ from \eqref{eq:dual-pairing-induced-isometry} is well-defined and bounded by the estimate
  \begin{equation*}
    \abs{\dualprod{\psi}{\phi}_{\Vtaud(\Gamma_{1}),\cVtau(\Gamma_{1})}} \leq \norm{\psi}_{\Vtaud(\Gamma_{1})} \norm{\phi}_{\cVtau(\Gamma_{1})}
  \end{equation*}
  from \Cref{le:continuous-dual-pairing}. In particular, $\norm{\Psi(\psi)}_{\cVtau(\Gamma_{1})\dual} \leq \norm{\psi}_{\Vtaud(\Gamma_{1})}$.
  Moreover,
  since the mapping $\tantr \colon \cH_{\Gamma_{0}}(\rot,\Omega) \to \cVtau(\Gamma_{1})$ is continuous, the composition $g \mapsto \mu(\tantr g)$  for fixed $\mu \in \cVtau(\Gamma_{1})\dual$ is continuous from $\cH_{\Gamma_{0}}(\rot,\Omega)$ to $\C$.
  This implies that there exists an $h \in \cH_{\Gamma_{0}}(\rot,\Omega)$ such that
  \begin{align*}
    \mu(\tantr g) = \scprod{h}{g}_{\Hspace(\rot,\Omega)} = \scprod{h}{g}_{\Lp{2}(\Omega)} + \scprod{\rot h}{\rot g}_{\Lp{2}(\Omega)}
  \end{align*}
  for all $g \in \cH_{\Gamma_{0}}(\rot,\Omega)$. In particular for $g \in \Cc(\Omega) \subseteq \ker \tantr$ we have
  \begin{equation*}
    0 = \scprod{h}{g}_{\Lp{2}(\Omega)} + \scprod{\rot h}{\rot g}_{\Lp{2}(\Omega)}
  \end{equation*}
  and consequently $\rot h \in \Hspace(\rot,\Omega)$ and $\rot \rot h = -h$. We define $\tilde{h} \coloneqq -\rot h$. Then for an arbitrary $g \in \cH_{\Gamma_{0}}(\rot,\Omega)$ we have
  \begin{align*}
    \mu(\tantr g)
    &= \scprod{h}{g}_{\Lp{2}(\Omega)} + \scprod{\rot h}{\rot g}_{\Lp{2}(\Omega)} \\
    &= \scprod{-\rot\rot h}{g}_{\Lp{2}(\Omega)} + \scprod{\rot h}{\rot g}_{\Lp{2}(\Omega)} \\
    &= \scprod{\rot(\underbrace{-\rot h}_{=\mathrlap{\tilde{h}}})}{g}_{\Lp{2}(\Omega)} - \scprod{\underbrace{-\rot h}_{=\mathrlap{\tilde{h}}}}{\rot g}_{\Lp{2}(\Omega)} \\
    &= \scprod{\rot \tilde{h}}{g}_{\Lp{2}(\Omega)} - \scprod{\tilde{h}}{\rot g}_{\Lp{2}(\Omega)} = \dualprod{\tanxtr \tilde{h}}{\tantr g}_{\Vtaud(\Gamma_{1}),\cVtau(\Gamma_{1})}.
  \end{align*}
  This implies $\mu = \Psi(\tanxtr \tilde{h})$ and therefore the surjectivity of $\Psi$.
  Hence, it it is left to show $\norm{\Psi(\psi)}_{\cVtau(\Gamma_{1})\dual} \geq \norm{\psi}_{\Vtaud(\Gamma_{1})}$.


  For $\psi \in \Vtaud(\Gamma_{1})$ we choose $f \in \Hspace(\rot,\Omega)$ as in \Cref{th:properties-of-preimage-of-Vtaud-elements}
  \begin{align*}
    \norm{\Psi(\psi)}_{\cVtau(\Gamma_{1})\dual}
    &= \sup_{\phi \in \cVtau(\Gamma_{1})\setminus\set{0}} \frac{\abs{\dualprod{\psi}{\phi}_{\Vtaud(\Gamma_{1}),\cVtau(\Gamma_{1})}}}{\norm{\phi}_{\cVtau(\Gamma_{1})}} \\
    &= \sup_{g \in \cH_{\Gamma_{0}}(\rot,\Omega)\setminus \ker \tantr*[\Gamma_{1}]}
      \frac{\abs{\dualprod{\tanxtr f}{\tantr g}_{\Vtaud(\Gamma_{1}),\cVtau(\Gamma_{1})}}}{\norm{\tantr g}_{\cVtau(\Gamma_{1})}} \\
    &\geq \sup_{g \in \cH_{\Gamma_{0}}(\rot,\Omega)\setminus \set{0}}
      \frac{\abs{\scprod{\rot f}{g}_{\Hspace(\rot,\Omega)}}}{\norm{g}_{\Hspace(\rot,\Omega)}} = \norm{\rot f}_{\Hspace(\rot,\Omega)} \\
    &= \norm{\psi}_{\Vtaud(\Gamma_{1})}
  \end{align*}
  which finishes the proof.
\end{proof}

Note that by construction our dual pair $(\cVtau(\Gamma_{1}),\Vtaud(\Gamma_{1}))$ is not only a dual pair, but its dual pairing is determined by an inner product, the inner product of the pivot space $\Lp{2}_{\tau}(\Gamma_{1})$. This special situation gives additional structure. In particular we call $(\cVtau(\Gamma_{1}),\Lp{2}_{\tau}(\Gamma_{1}),\Vtaud(\Gamma_{1}))$ a \emph{quasi Gelfand triple}, see \cite{Sk21,Sk23} for a detailed discussion of this notion. Roughly speaking a quasi Gelfand triple is a Gelfand triple without continuous embeddings, but instead ``closed embeddings''.
\Cref{fig:comparism-ordinary-gelfand-triples-to-quasi-gelfand-triple} illustrates the difference to ``ordinary'' Gelfand triples.

\begin{figure}[h]
  \centering
  \begin{subfigure}{0.49\textwidth}
    \centering
    \begin{tikzpicture}
      \draw[thick](-6,0) circle [radius=1.6];
      \draw[thick,blue](-6,0.2) circle [radius=1.3];
      \draw[thick,black!30!green](-6,-0.2) circle [radius=1.9];

      \small
      \node[above] at (-6,-1.6) {$\Lp{2}(\partial\Omega)$};
      \node[blue,above] at (-6,-1) {$\Hspace^{\nicefrac{1}{2}}\mspace{-1mu}(\partial\Omega)$};
      \node[black!30!green,above] at (-6,-2.1) {$\Hspace^{-\nicefrac{1}{2}}\mspace{-1mu}(\partial\Omega)$};

    \end{tikzpicture}
    \caption{Gelfand triple}
  \end{subfigure}
  \begin{subfigure}{0.49\textwidth}
    \centering
    \begin{tikzpicture}
      \begin{scope}
        \clip (-0.5,-0.5) circle (1.5);
        \fill[blue,opacity=0.15] (0,0) circle (1.5);
      \end{scope}
      \begin{scope}
        \clip (0.5,-0.5) circle (1.5);
        \fill[black!30!green,opacity=0.15] (0,0) circle (1.5);
      \end{scope}
      \draw[thick](0,0) circle [radius=1.5];
      \draw[thick,blue](-0.5,-0.5) circle [radius=1.5];
      \draw[thick,black!30!green](0.5,-0.5) circle [radius=1.5];
      \node[left,blue] at (-1.7,-1.5) {$\Vtau(\partial\Omega)$};
      \node[right,black!30!green] at (1.7,-1.5) {$\Vtaud(\partial\Omega)$};
      \node[above] at (0,1.5) {$\Lp{2}_{\tau}(\partial\Omega)$};
    \end{tikzpicture}
    \caption{Quasi Gelfand triple}
  \end{subfigure}

  \caption{Difference between ordinary and quasi Gelfand triples}
  \label{fig:comparism-ordinary-gelfand-triples-to-quasi-gelfand-triple}
\end{figure}
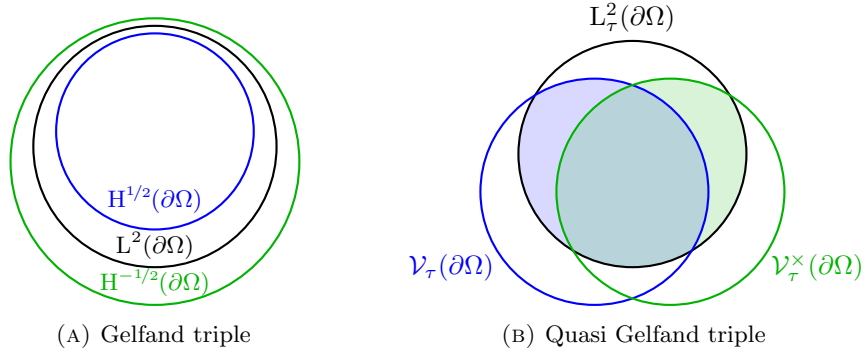


\subsection{L\textsuperscript{2} tangential traces}

In order to define the set $\hH_{\Gamma_{1}}(\rot,\Omega)$ properly we have to explain what we mean by $\tantr f \in \Lp{2}(\Gamma_{1})$.
First, we want to recall that the space $\cH_{\Gamma_{0}}(\rot,\Omega)$ is defined as the set of all $f \in \Hspace(\rot,\Omega)$ that satisfy
\begin{equation*}
  \scprod{\rot f}{\varphi}_{\Lp{2}(\Omega)} - \scprod{f}{\rot \varphi}_{\Lp{2}(\Omega)} = 0 \quad\text{for all}\quad \varphi \in \Cc_{\Gamma_{1}}(\R^{3}).
\end{equation*}
Note that, as we have already used, this set coincides with the closure of $\cH^{1}_{\Gamma_{0}}(\Omega)$ in $\Hspace(\rot,\Omega)$, see \cite{BaPaScho16}.

\smallskip
Finally we define what it means for an $f \in \Hspace(\rot,\Omega)$ to have an $\Lp{2}(\Gamma_{1})$ tangential trace. There are two a priori different approaches: a weak approach by representation in an inner product and a strong approach by convergence.

\begin{definition}[weak $\Lp{2}_{\tau}(\Gamma_{1})$ tangential trace]
  \phantom{a}
  \begin{itemize}[leftmargin=\parindent,itemsep=3pt,parsep=2pt]
    \item For $f \in \Hspace(\rot,\Omega)$ we say $\tanxtr f$ is weakly in $\Lp{2}_{\tau}(\Gamma_{1})$, if there exists an $h \in \Lp{2}_{\tau}(\Gamma_{1})$ such that
          \begin{equation*}
            \scprod{\rot f}{\varphi}_{\Lp{2}(\Omega)} - \scprod{f}{\rot \varphi}_{\Lp{2}(\Omega)} = \scprod{h}{\tantr \varphi}_{\Lp{2}(\Gamma_{1})}
            \quad\text{for all}\quad \varphi \in \Cc_{\Gamma_{0}}(\R^{3}).
          \end{equation*}
          We then say $\tanxtr f = h$ weakly (on $\Gamma_{1}$).

    \item For $g \in \Hspace(\rot,\Omega)$ we say $\tantr g$ is weakly in $\Lp{2}_{\tau}(\Gamma_{1})$, if there exists an $h \in \Lp{2}_{\tau}(\Gamma_{1})$ such that
          \begin{equation*}
            \scprod{\rot \varphi}{g}_{\Lp{2}(\Omega)} - \scprod{\varphi}{\rot g}_{\Lp{2}(\Omega)} = \scprod{\tanxtr \varphi}{h}_{\Lp{2}(\Gamma_{1})}
            \quad\text{for all}\quad \varphi \in \Cc_{\Gamma_{0}}(\R^{3}).
          \end{equation*}
          We then say $\tantr g = h$ weakly (on $\Gamma_{1}$).
  \end{itemize}
\end{definition}

\begin{definition}[strong $\Lp{2}_{\tau}(\Gamma_{1})$ tangential trace]
  \phantom{a}
  \begin{itemize}[leftmargin=\parindent,itemsep=3pt,parsep=2pt]
    \item For $f \in \Hspace(\rot,\Omega)$ we say $\tanxtr f$ is strongly in $\Lp{2}_{\tau}(\Gamma_{1})$,
          if there exists a sequence $(\varphi_{n})_{n\in\N}$ in $\Cc(\R^{3})$ and an $h \in \Lp{2}_{\tau}(\Gamma_{1})$ such that
          \begin{equation*}
            \norm{f - \varphi_{n}}_{\Hspace(\rot,\Omega)} \to 0 \quad\text{and}\quad \norm{\tanxtr \varphi_{n} - h}_{\Lp{2}(\Gamma_{1})} \to 0.
          \end{equation*}
          We then say $\tanxtr f = h$ strongly (on $\Gamma_{1}$).

    \item For $g \in \Hspace(\rot,\Omega)$ we say $\tantr g$ is strongly in $\Lp{2}_{\tau}(\Gamma_{1})$,
          if there exists a sequence $(\varphi_{n})_{n\in\N}$ in $\Cc(\R^{3})$ and an $h \in \Lp{2}_{\tau}(\Gamma_{1})$ such that
          \begin{equation*}
            \norm{g - \varphi_{n}}_{\Hspace(\rot,\Omega)} \to 0 \quad\text{and}\quad \norm{\tantr \varphi_{n} -  h}_{\Lp{2}(\Gamma_{1})} \to 0.
          \end{equation*}
          We then say $\tantr g = h$ strongly (on $\Gamma_{1}$).
  \end{itemize}
\end{definition}

Note that it is not really necessary to have separate definition for $\tantr f \in \Lp{2}_{\tau}(\Gamma_{1})$ and $\tanxtr f \in \Lp{2}_{\tau}(\Gamma_{1})$, as they are actually equivalent (both weakly and strongly, respectively). In particular, if one of these tangential trace is $\Lp{2}_{\tau}(\Gamma_{1})$, then we can calculate the other by the following formula $\tantr f = \tanxtr f \times \nu$ or equivalently $\nu \times \tantr f = \tanxtr f$.



Note that we can replace $\varphi \in \Cc_{\Gamma_{0}}(\R^{3})$ by $\varphi \in \cH^{1}_{\Gamma_{0}}(\Omega)$ (by a density argument) in the weak definition.

\begin{remark}
  The previous definitions can also be formulated in terms of the boundary spaces $\Vtaud(\Gamma_{1})$ and $\cVtau(\Gamma_{1})$:
  \begin{itemize}
    \item
          We say $\psi \in \Vtaud(\Gamma_{1})$ is weakly in the pivot space $\Lp{2}_{\tau}(\Gamma_{1})$, if there exists an $h\in \Lp{2}_{\tau}(\Gamma_{1})$ such that
          \begin{align*}
            \dualprod{\psi}{\phi}_{\Vtaud(\Gamma_{1}),\cVtau(\Gamma_{1})} = \scprod{h}{\phi}_{\Lp{2}_{\tau}(\Gamma_{1})} \quad\text{for all}\quad \phi \in \mathring{M}(\Gamma_{1}).
          \end{align*}
          Hence, $\tanxtr f \in \Lp{2}(\Gamma_{1})$ weakly, if
          \begin{align*}
            \dualprod{\tanxtr f}{\tantr g}_{\Vtaud(\Gamma_{1}),\cVtau(\Gamma_{1})} = \scprod{h}{\tantr g}_{\Lp{2}_{\tau}(\Gamma_{1})} \quad\text{for all}\quad g \in \cH^{1}_{\Gamma_{0}}(\Omega).
          \end{align*}
          Clearly, we can do the same for $\Vtau(\Gamma_{1})$ and $\tantr$.

    \item
          We say $\psi \in \Vtaud(\Gamma_{1})$ is strongly in $\Lp{2}_{\tau}(\Gamma_{1})$, if there exists a sequence $(\psi_{n})_{n\in\N}$ in $M^{\times}(\Gamma_{1})$ and an $h \in \Lp{2}_{\tau}(\Gamma_{1}) $ such that
          \begin{equation*}
            \norm{\psi_{n} - \psi}_{\Vtaud(\Gamma_{1})} \to 0 \quad\text{and}\quad \norm{\psi_{n} - h}_{\Lp{2}_{\tau}(\Gamma_{1})} \to 0.
          \end{equation*}
          Hence, $\tanxtr f \in \Lp{2}(\Gamma_{1})$ strongly, if there exists a sequence $(f_{n})_{n\in\N}$ in $\Hspace^{1}(\Omega)$ and an $h \in \Lp{2}_{\tau}(\Gamma_{1})$ such that
          \begin{equation*}
            \norm{\tanxtr f_{n} - \tanxtr f}_{\Vtaud(\Gamma_{1})} \to 0 \quad\text{and}\quad \norm{\tanxtr f_{n} - h}_{\Lp{2}_{\tau}(\Gamma_{1})} \to 0.
          \end{equation*}
          Note that $\tanxtr\colon (\ker \tanxtr[\Gamma_{1}])^{\perp} \to \Vtaud(\Gamma_{1})$ is a unitary mapping. Therefore, we can modify $(f_{n})_{n\in\N}$ such that we also have $\norm{f_{n} - f}_{\Hspace(\rot,\Omega)} \to 0$.\footnote{If we really want to prove this modification of $f_{n}$ such that the other properties are preserved, we would also use $\ker \tanxtr[\Gamma_{1}] = \cl[\Hspace(\rot,\Omega)]{\ker \tanxtr[\Gamma_{1}] \cap \Hspace^{1}(\Omega)}$.}
  \end{itemize}

\end{remark}

Since we have a proper definition for $\tantr f \in \Lp{2}_{\tau}(\Gamma_{1})$ we can finally properly define the space $\hH_{\Gamma_{1}}(\rot,\Omega)$.

\begin{definition}
  We define
  \begin{align*}
    \hH_{\Gamma_{1}}(\rot,\Omega)
    &\coloneqq \dset{f \in \Hspace(\rot,\Omega)}{\tantr f \in \Lp{2}_{\tau}(\Gamma_{1}) \ \text{weakly}} \\
    &\phantom{\mathclose{}\coloneqq\mathopen{}}\mathllap{=\mathopen{}}
    \dset{f \in \Hspace(\rot,\Omega)}{\tanxtr f \in \Lp{2}_{\tau}(\Gamma_{1}) \ \text{weakly}}
  \end{align*}
  and endow this set with the ``natural'' inner product
  \begin{align*}
    \scprod{f}{g}_{\hH_{\Gamma_{1}}(\rot,\Omega)}
    &\coloneqq \scprod{f}{g}_{\Lp{2}(\Omega)} + \scprod{\rot f}{\rot g}_{\Lp{2}(\Omega)} + \scprod{\tantr f}{\tantr g}_{\Lp{2}_{\tau}(\Gamma_{1})} \\
    &\phantom{\mathclose{}\coloneqq\mathopen{}}\mathllap{=\mathopen{}}
      \scprod{f}{g}_{\Lp{2}(\Omega)} + \scprod{\rot f}{\rot g}_{\Lp{2}(\Omega)} + \scprod{\tanxtr f}{\tanxtr g}_{\Lp{2}_{\tau}(\Gamma_{1})}.
  \end{align*}
\end{definition}

Note that for $f \in \Hspace(\rot,\Omega)$ it is straightforward to show, that $\tantr f$ is strongly in $\Lp{2}_{\tau}(\Gamma_{1})$ implies it is also weakly in $\Lp{2}_{\tau}(\Gamma_{1})$ and the weak $\Lp{2}_{\tau}(\Gamma_{1})$ trace equals the strong one.
Hence, the set of all $f \in \Hspace(\rot,\Omega)$ that satisfy $\tantr f \in \Lp{2}_{\tau}(\Gamma_{1})$ strongly is a subspace of $\hH_{\Gamma_{1}}(\rot,\Omega)$. In particular it is the closure of $\Hspace^{1}(\Omega)$ in $\hH_{\Gamma_{1}}(\rot,\Omega)$ w.r.t.\ $\norm{\cdot}_{\hH_{\Gamma_{1}}(\rot,\Omega)}$. The question that immediately arises is: ``Do these sets coincide?'' or equivalently
\begin{center}
  ``Is $\Hspace^{1}(\Omega)$ dense in $\hH_{\Gamma_{1}}(\rot,\Omega)$ (w.r.t.\ $\norm{\cdot}_{\hH_{\Gamma_{1}}(\rot,\Omega)}$)?''
\end{center}
For $\Gamma_{1} = \partial\Omega$ this question was answered in \cite{BeBeCoDa97}. For general $\Gamma_{1}$ this question is formulated as an open problem in \cite{WeSt13} (to be precise the dual version of this question was asked in \cite{WeSt13}, which we will formulate later in \eqref{eq:dual-density-question}) and was recently answered in \cite{SkPa23}.

\medskip

If we want to characterize $\Lp{2}$ tangential traces on $\Gamma_{1}$ for functions that additionally have a homogeneous tangential trace on $\Gamma_{0}$, then it turns out to be more convenient to regard this in a combined way. More precisely, we do not look at the space $\hH_{\Gamma_{1}}(\rot,\Omega)\cap \cH_{\Gamma_{0}}(\rot,\Omega)$, but at $\hH_{\partial\Omega}(\rot,\Omega) \cap \cH_{\Gamma_{0}}(\rot,\Omega)$. This means that we regard functions that have an $\Lp{2}$ tangential trace on the entire boundary and on one part of the boundary this tangential trace vanishes. One might hope that these spaces coincide, but it is even unclear to us, whether
\begin{equation}\label{eq:new-open-problem}
  \cH_{\Gamma_{0}}(\rot,\Omega) \cap \cH_{\Gamma_{1}}(\rot,\Omega) = \cH_{\partial\Omega}(\rot,\Omega)
\end{equation}
holds true. Hence, these questions can be seen as \textbf{open problems}.

\medskip

Anyway, for our purpose it suffices to work with $\hH_{\partial\Omega}(\rot,\Omega) \cap \cH_{\Gamma_{0}}(\rot,\Omega)$. A strong approach to that space would be to regard limits of $\cH_{\Gamma_{0}}^{1}(\Omega)$ elements w.r.t.\ $\norm{\cdot}_{\hH_{\partial\Omega}(\rot,\Omega)}$, which is the closure of $\cH_{\Gamma_{0}}^{1}(\Omega)$ in $\hH_{\partial\Omega}(\rot,\Omega) \cap \cH_{\Gamma_{0}}(\rot,\Omega)$. Again the question that arises is:
\begin{equation}\label{eq:dual-density-question}
  \text{``Is $\cH_{\Gamma_{0}}^{1}(\Omega)$ dense in $\hH_{\partial\Omega}(\rot,\Omega) \cap \cH_{\Gamma_{0}}(\rot,\Omega)$ (w.r.t.\ $\norm{\cdot}_{\hH_{\partial\Omega}(\rot,\Omega)}$)?''}
\end{equation}
This is the actual formulation of the open problem of \cite[eq.~(5.20)]{WeSt13} and at the end of section~5 in \cite{WeSt13}. Also this question was answered recently in \cite{SkPa23}.

Hence, long story short, we do not have to distinguish between strongly and weakly in $\Lp{2}_{\tau}(\Gamma_{1})$ in neither case. This is in particular of advantage, because we want to formulate our boundary conditions in $\Lp{2}_{\tau}(\Gamma_{1})$ and the framework of quasi Gelfand triples would regard one of $\cVtau(\Gamma_{1}) \cap \Lp{2}_{\tau}(\Gamma_{1})$ and $\Vtaud(\Gamma_{1}) \cap \Lp{2}_{\tau}(\Gamma_{1})$ strongly and the other one -- for duality reasons -- weakly. Therefore, now we avoid to have two different concepts of $\Lp{2}_{\tau}(\Gamma_{1})$ tangential traces.

\subsection{Contraction semigroup}\label{sec:generator}

Finally, we want to consider the actual question of this section.
It lies in the nature of block operators like
\(
\begin{bsmallmatrix}
  0 & \rot \\
  -\rot & 0
\end{bsmallmatrix}
\),
that if we define the domain of the upper right block strongly, then we have to define the domain of the (adjoint) lower left block weakly. The same goes for the corresponding traces. In this particular case, where both operators are basically the same, this leads to the strange situation where we have to introduce two a priori different definitions for the domain of the same operator. Luckily as we have discussed in this section both approaches coincide and consequently we just need one definition.

In order to justify that $A_{0}$ generates a contraction semigroup we want to use \cite[Thm.~5.3.6]{Sk-Phd} or originally \cite[Thm.~7.6 \& Ex.~7.8]{Sk21}. For convenience we provide \cite[Thm.~5.3.6]{Sk-Phd} as \Cref{th:generator-of-contraction-semigroup}. In order to understand how this theorem is applicable, we have to translate our setting into their notation. In particular our setting is a special case of a general theory.
As we have remarked at the beginning of this section our operator fits into the framework of \cite{Sk21}, because we can decompose the $\rot$ operator into $\rot = \sum_{i=1}^{3} \partial_{i} L_{i} = \diffop$, as we did in~\eqref{eq:decompose-rot}. The corresponding Hermitian transposed operator $\diffopad$ is in our case (by the skew-adjointness of $L_{i}$)
\begin{equation*}
  \diffopad \coloneqq \sum_{i=1} \partial_{i} L_{i}\hermitian = - \rot.
\end{equation*}
Consequently, the block operator $P_{\partial}$ that is regarded by the theory in \cite{Sk21,Sk-Phd} matches our block differential operator
\begin{equation*}
  P_{\partial} \coloneqq
  \begin{bmatrix}
    0 & \diffop \\
    \diffopad & 0
  \end{bmatrix}
  =
  \begin{bmatrix}
    0 & \rot \\
    -\rot & 0
  \end{bmatrix}.
\end{equation*}
In our case the additional $P_{0}$ is just $0$. For the boundary operators and boundary space we have the following translation
\begin{equation*}
  \pi_{L} = \tantr, \quad L_{\nu} = \tanxtr \quad\text{and}\quad \Lp{2}_{\pi}(\Gamma_{1}) = \Lp{2}_{\tau}(\Gamma_{1}).
\end{equation*}

\begin{theorem}[{\cite[Thm.~5.3.6]{Sk-Phd}}]\label{th:generator-of-contraction-semigroup}
  Let $M$ be a strictly positive linear operator on $\Lp{2}_{\pi}(\Gamma_{1})$. Then the operator $A = (P_{\partial} + P_{0}) \hamiltonian$ with domain
  \begin{align*}
    \dom A
    =
    \dset*{x \in \hamiltonian^{-1} \begin{bmatrix} \cH_{\Gamma_{0}}(\diffopad,\Omega) \\ \Hspace(\diffop,\Omega)\end{bmatrix}}{\pi_{L} (\hamiltonian x)_{L\hermitian} + M L_{\nu} (\hamiltonian x)_{L} = 0 \ \text{in}\ \Lp{2}_{\pi}(\Gamma_{1})}
  \end{align*}
  generates a contraction semigroup, where $\hamiltonian x = \begin{bsmallmatrix}(\hamiltonian x)_{L\hermitian} \\ (\hamiltonian x)_{L}\end{bsmallmatrix}$.
\end{theorem}

Note that the previous theorem applied to our situation has the allegedly weaker condition
\begingroup
\renewcommand{\quad}{\mspace{8mu}}
\begin{equation*}
  \begin{bmatrix}
    \vec{D} \\ \vec{B}
  \end{bmatrix}
  \in
  \hamiltonian^{-1}
  \begin{bmatrix}
    \cH_{\Gamma_{0}}(\rot,\Omega) \\
    \Hspace(\rot,\Omega)
  \end{bmatrix}
  \quad\text{instead of}\quad
  \begin{bmatrix}
    \vec{D} \\ \vec{B}
  \end{bmatrix}
  \in
  \hamiltonian^{-1}
  \begin{bmatrix}
    \hH_{\partial\Omega}(\rot,\Omega) \cap \cH_{\Gamma_{0}}(\rot,\Omega) \\
    \hH_{\Gamma_{1}}(\rot,\Omega)
  \end{bmatrix}.
\end{equation*}
\endgroup
However, since the boundary condition is formulated in $\Lp{2}_{\tau}(\Gamma_{1})$, we can also integrate this directly into the space and obtain our original condition.\footnote{To be precise---and avoid the open problem before \eqref{eq:new-open-problem}---, being in $\Lp{2}_{\pi}(\Gamma_{1})$ in \Cref{th:generator-of-contraction-semigroup} has to be understood in the sense of being in the pivot space of the corresponding quasi Gelfand triple, which automatically avoids this possible ambiguity.}
Therefore, we finally conclude the following corollary.

\begin{corollary}
  The operator $A_{0}$ from \eqref{eq:def-A-0} generates a contraction semigroup.
\end{corollary}

Alternatively, it was also shown in \cite{WeSt13} that $A_{0}$ generates a contraction semigroup.

Note for elements in the domain of $A_{0}$ we can apply the extended integration by parts formula \Cref{le:int-by-parts-extended-to-max-dom} and replace the dual pairing by an $\Lp{2}_{\tau}(\Gamma_{1})$ inner product as all elements in $\dom A_{0}$ have $\Lp{2}$ tangential traces. Hence, applying this integration by parts formula twice for both rows of $A_{0}$ gives the following lemma.

\begin{lemma}\label{le:int-by-parts-for-A}
  For
  \(
  x =
  \begin{bsmallmatrix}
    \vec{D}_{x} \\ \vec{B}_{x}
  \end{bsmallmatrix}
  ,
  y =
  \begin{bsmallmatrix}
    \smash[b]{\vec{D}_{y}} \\ \smash[b]{\vec{B}_{y}}
  \end{bsmallmatrix}
  \in \dom A_{0}
  \)
  we have
  \begin{align*}
    \scprod{A_{0}x}{y}_{\hamiltonian} + \scprod{x}{A_{0}y}_{\hamiltonian}
    = \scprod*{\tantr \epsilon^{-1} \vec{D}_{x}}{\tanxtr \mu^{-1} \vec{B}_{y}}_{\Lp{2}(\Gamma_{1})}
    + \scprod*{\tanxtr \mu^{-1} \vec{B}_{x}}{\tantr \epsilon^{-1} \vec{D}_{y}}_{\Lp{2}(\Gamma_{1})}
  \end{align*}
  and in particular
  \begin{align*}
    \Re \scprod{A_{0}x}{x}_{\hamiltonian} = \Re \scprod*{\tantr \epsilon^{-1} \vec{D}_{x}}{\tanxtr \mu^{-1} \vec{B}_{x}}_{\Lp{2}(\Gamma_{1})}.
  \end{align*}
\end{lemma}

\section*{Acknowledgement}
We thank Dirk Pauly for making us aware of the unique continuation principle for Maxwell's equations.

\bibliographystyle{abbrvurl}

\bibliography{bibfile}{}

\end{document}